\documentclass[a4paper,11pt]{article}
\usepackage[latin1]{inputenc}
\usepackage[english]{babel}
\usepackage{amsmath}
\usepackage{amsfonts}
\usepackage{amssymb}
\usepackage{epsfig}
\usepackage{amsopn}
\usepackage{amsthm}
\usepackage{color}
\usepackage{graphicx}
\usepackage{subfigure}
\usepackage{enumerate}
\newtheorem{theorem}{Theorem}[section]

\newtheorem{lemma}[theorem]{Lemma}
\newtheorem{proposition}[theorem]{Proposition}

\newtheorem*{theorem*}{Theorem}
\newtheorem*{lemma*}{Lemma}
\newtheorem*{remark*}{Remark}
\newtheorem*{definition*}{Definition}
\newtheorem*{proposition*}{Proposition}
\newtheorem*{corollary*}{Corollary}
\numberwithin{equation}{section}
\newcommand{\real}{\mathbb{R}}

\let\ced=\c         

\def\qed{\,\unskip\kern 6pt \penalty 500
\raise -2pt\hbox{\vrule \vbox to8pt{\hrule width 6pt
\vfill\hrule}\vrule}\par}
\definecolor{darkblue}{rgb}{0.05, .05, .65}
\definecolor{darkgreen}{rgb}{0.1, .65, .1}
\definecolor{darkred}{rgb}{0.8,0,0}
\newcommand{\beqn}{\begin{equation}}
\newcommand{\eeqn}{\end{equation}}
\newcommand{\bear}{\begin{eqnarray}}
\newcommand{\eear}{\end{eqnarray}}
\newcommand{\bean}{\begin{eqnarray*}}
\newcommand{\eean}{\end{eqnarray*}}
\begin{document}

\title{\huge \bf Anomalous self-similar solutions of exponential type for the subcritical fast diffusion equation with weighted reaction}

\author{
\Large Razvan Gabriel Iagar\,\footnote{Departamento de Matem\'{a}tica
Aplicada, Ciencia e Ingenieria de los Materiales y Tecnologia
Electr\'onica, Universidad Rey Juan Carlos, M\'{o}stoles,
28933, Madrid, Spain, \textit{e-mail:} razvan.iagar@urjc.es},
\\[4pt] \Large Ariel S\'{a}nchez,\footnote{Departamento de Matem\'{a}tica
Aplicada, Ciencia e Ingenieria de los Materiales y Tecnologia
Electr\'onica, Universidad Rey Juan Carlos, M\'{o}stoles,
28933, Madrid, Spain, \textit{e-mail:} ariel.sanchez@urjc.es}\\
[4pt] }
\date{}
\maketitle

\begin{abstract}
We prove existence and uniqueness of the branch of the so-called \emph{anomalous eternal solutions} in exponential self-similar form for the subcritical fast-diffusion equation with a weighted reaction term
$$
\partial_tu=\Delta u^m+|x|^{\sigma}u^p,
$$
posed in $\real^N$ with $N\geq3$, where
$$
0<m<m_c=\frac{N-2}{N}, \qquad p>1,
$$
and the critical value for the weight
$$
\sigma=\frac{2(p-1)}{1-m}.
$$
The branch of exponential self-similar solutions behaves similarly as the well-established anomalous solutions to the pure fast diffusion equation, but without a finite time extinction or a finite time blow-up, and presenting instead a \emph{change of sign of both self-similar exponents} at $m=m_s=(N-2)/(N+2)$, leading to surprising qualitative differences. In this sense, the reaction term we consider realizes a \emph{perfect equilibrium} in the competition between the fast diffusion and the reaction effects.
\end{abstract}

\


\smallskip


\section{Introduction}

The goal of this paper is to establish the existence and uniqueness of \emph{eternal solutions} in exponential self-similar form for the fast diffusion equation with weighted reaction
\begin{equation}\label{eq1}
\partial_tu=\Delta u^m+|x|^{\sigma}u^p,
\end{equation}
posed in $\real^N$, with $m<1$, $p>1$, $N\geq3$ and the critical value of the exponent $\sigma$
\begin{equation}\label{sig}
\sigma=\frac{2(p-1)}{1-m}.
\end{equation}
As we shall see in the paper, we construct a branch of self-similar solutions which are analogous to the celebrated \emph{anomalous solutions} introduced by \cite{Ki93, PZ95}, for the \emph{subcritical range} of the fast diffusion
\begin{equation}\label{mc}
0<m<m_c:=\frac{N-2}{N}, \qquad N\geq3,
\end{equation}
showing at the same time that such solutions \emph{do no longer exist} in the complementary range $m_c\leq m<1$.

The main mathematical feature of Eq. \eqref{eq1} is the \emph{competition} between the two effects that appear in its formulation: a fast diffusion which, in the subcritical range \eqref{mc}, tends to finite time extinction due to a loss of mass through infinity which is well explained in \cite[Section 5.5, p. 91]{VazSmooth}, and a weighted reaction tending to add mass to the solution (and thus compensate for the loss of mass explained above) and whose sharpest effect is in many occasions the appearance of finite time blow-up. One of the most interesting outcomes of the current work is that the value of $\sigma$ in \eqref{sig} leads to a \emph{perfect equilibrium} between these two effects: as we shall see, the solutions we construct to Eq. \eqref{eq1} do neither extinguish, nor blow up in finite time.

The fast diffusion equation
\begin{equation}\label{FDE}
u_t=\Delta u^m, \qquad 0<m<1
\end{equation}
is by now a well studied equation and proved to be a very interesting object for research due to a number of unexpected properties and effects. A rather complete monograph on it \cite{VazSmooth} is available nowadays. In particular, unusual mathematical behaviors appear in the so-called subcritical range $0<m<m_c$, with $m_c$ introduced in \eqref{mc}. In this range, finite time extinction takes place at least for integrable initial conditions, and it was a quest for establishing the dynamics near the extinction time that led to the so-called \emph{special solutions with anomalous exponents} (or shortly \emph{anomalous solutions}) noticed formally in \cite{Ki93} and introduced rigorously by Peletier and Zhang \cite{PZ95}. These are solutions to Eq. \eqref{FDE} in backward self-similar form
\begin{equation}\label{anomalous}
U(x,t)=(T-t)^{\alpha(m)}f(|x|(T-t)^{\beta(m)}, \qquad {\rm with} \ (1-m)\alpha(m)=2\beta(m)+1,
\end{equation}
where $T>0$ is the extinction time and the profile $f$ satisfies the following decay rate
\begin{equation}\label{decay.anom}
f(\xi)\sim C\xi^{-(N-2)/m}, \qquad {\rm as} \ \xi\to\infty
\end{equation}
The branch of self-similar exponents $\alpha(m)>0$ (and thus also $\beta(m)$) and the profile $f$ with behavior at infinity as in \eqref{decay.anom} are shown to be unique for a fixed $m\in(0,m_c)$ in \cite{PZ95} or \cite[Section 7.2]{VazSmooth}. These solutions were named \emph{anomalous} since, in contrast with other well-established self-similar solutions, their self-similar exponents are not explicit and not obtained through an algebraic calculation, but through a dynamical system technique (phase plane analysis) leading also to the existence and uniqueness of the profile. These solutions proved to be of an utmost importance for the dynamics of the equation: it was shown by Galaktionov and Peletier \cite{GP97} that for any \emph{radially symmetric initial condition} $u_0$ with suitable regularity and decaying sufficiently fast (more precisely, such that $u_0(r)=O(r^{-q})$ as $r\to\infty$ for some $q>2/(1-m)$), the unique solution to Eq. \eqref{FDE} having initial condition $u_0$ approaches the anomalous solution $U(x,t)$ in \eqref{anomalous} near the finite extinction time. We can thus say that the special solutions in \eqref{anomalous} describe the asymptotic behavior of Eq. \eqref{FDE} for radial solutions. With respect to non-radial solutions, the problem of the asymptotic behavior is more complicated, and it was established in \cite{dPS01} that, for the particular case of another important exponent in the theory of fast diffusion, the \emph{Sobolev exponent}
\begin{equation}\label{ms}
m_s=\frac{N-2}{N+2}\in(0,m_c),
\end{equation}
there are non-radial asymptotic profiles. Up to our knowledge, a full description of all possible asymptotic profiles of Eq. \eqref{FDE} with non-radially symmetric data is still an open problem for $m\in(0,m_c)$ with $m\neq m_s$. However, the only case when the anomalous solution is explicit is exactly for $m=m_s$ (a case related to the \emph{Yamabe flow} in Riemannian geometry), with exponents
$$
\beta(m_s)=0, \ \alpha(m_s)=\frac{1}{1-m_s}=\frac{N+2}{4}.
$$

The second effect, competing with the fast diffusion in the dynamics of Eq. \eqref{eq1}, is the (weighted) reaction. The standard (homogeneous) reaction-diffusion equation
\begin{equation}\label{eq1.hom}
u_t=\Delta u^m+u^p, \qquad m>0, \ p>1
\end{equation}
has been thoroughly investigated for the slow diffusion case $m>1$. The main feature of this equation is the blow-up in finite time of its solutions, that is, the existence of a time $T\in(0,\infty)$ such that $u(t)\in L^{\infty}(\real^N)$ for $t\in(0,T)$ but $u(T)\not\in L^{\infty}(\real^N)$. For $m>1$, many properties of solutions to Eq. \eqref{eq1.hom} are known, including when finite time blow-up takes place, blow-up rates and profiles \cite{S4}. In particular, a relevant fact is that there exists a critical exponent known as \emph{the Fujita exponent}
\begin{equation}\label{Fujita}
p_F=m+\frac{2}{N}
\end{equation}
such that for $m>1$ and any $p\in(1,p_F)$, all the non-trivial solutions to Eq. \eqref{eq1.hom} blow up in finite time. Concerning the fast diffusion range $m\in(0,1)$ in Eq. \eqref{eq1.hom}, it is established in \cite{Qi93, MM95} that for $m_c<m<1$, the exponent $p_F>1$ given in \eqref{Fujita} still plays the role of a Fujita-type exponent in the sense described above: for $1<p<p_F$ any non-trivial solution still blows up in finite time. Later on, Guo and Guo \cite{GuoGuo01} studied the range $p>p_F$ when global solutions may exist, establishing the required decay rate as $|x|\to\infty$ of the initial condition in order for the solution to be global, and giving the large time behavior of these global solutions. Maing\'e \cite{Ma08} extends the results related to the connection between the decay rate as $|x|\to\infty$ of the data $u_0$ and finite time blow-up to the whole fast diffusion range $m\in(0,1)$. Understanding the blow-up behavior of solutions when the reaction term is weighted, that is, for equations such as
\begin{equation}\label{eq1.gen}
u_t=\Delta u^m+a(x)u^p, \qquad m>1
\end{equation}
was a problem addressed since long and studied (for suitable weights $a(x)$, not necessarily pure powers) in papers as \cite{BK87, BL89, Pi97, Pi98}. We quote here the paper by Suzuki \cite{Su02} where the pure power weight $a(x)=|x|^{\sigma}$, $\sigma>0$ is considered and both the Fujita-type exponent and the second critical exponent related to blow-up (the critical decay of $u_0(x)$ as $|x|\to\infty$ when $p>p_F$ splitting between blow-up solutions and global solutions) are given, we recall here the first of them:
\begin{equation}\label{Fujita_sigma}
p_{F,\sigma}=m+\frac{2+\sigma}{N}.
\end{equation}
Recently, the authors started a long-term program of understanding and classifying the blow-up profiles to equations such as Eq. \eqref{eq1} and a number of results related to the range $m>1$ have been obtained in a series of papers \cite{IS19, IS20b, IS21, IS21c}. In these papers, considering always $m>1$, the blow-up profiles in the form of backward self-similar solutions are classified for reaction exponents $p\in[1,m]$ and any $\sigma>0$, and the results were sometimes rather unexpected and strongly depending both on the relation between $m$ and $p$ and on the magnitude of $\sigma$. In all these cases, the blow-up profiles are compactly supported, presenting interfaces, but the behavior at $x=0$ may vary and, what is most important, \emph{the blow-up set varies} with $\sigma$. In particular, a fact that should be emphasized is that, in many cases, even $x=0$ is a blow-up point, despite the fact that, formally, there is no reaction at all at the origin. However, we have noticed that fast diffusion with weighted reaction has been considered only very seldom. Qi \cite{Qi98} considers a reaction with a weight including also a time dependence
$$
u_t=\Delta u^m+t^s|x|^{\sigma}u^p, \qquad m>m_c, \ s\geq0, \ \sigma>-2
$$
and proves that there exists a Fujita-type exponent with an explicit expression depending on $m$, $s$ and $\sigma$ which reduces to $p_{F,\sigma}$ for $s=0$. Later, localized weights $a(x)$ with compact support have been considered in Eq. \eqref{eq1.gen} and analyzed in \cite{BZZ11}, noticing that the Fujita-type exponents changes into $p_F=m+1$.

After this discussion of the two effects present in Eq. \eqref{eq1} and of the precedents of the problem, let us get closer to the main contributions of this paper.

\medskip

\noindent \textbf{Main results.} The previous detailed discussion about precedents shows that we are dealing in Eq. \eqref{eq1} with a \emph{competition} between two terms generating typically \emph{two totally opposite effects}: on the one hand the diffusion tends to finite time extinction (that is, mass tending to zero), and on the other hand the reaction typically tends to finite time blow-up (mass tending to infinity). It is thus an interesting question to find a combination of exponents giving a \emph{perfect balance} between them, leading to dynamics that do not either vanish or blow up in finite time. This is achieved by restricting ourselves to the subcritical range $m\in(0,m_c)$ and letting $\sigma$ as in \eqref{sig} and then obtaining that Eq. \eqref{eq1} has so-called \emph{eternal solutions}, that is, global solutions having an exponential dependence on the time variable and which can be defined even for any $t\in\real$ (that is, also backward in time). The terminology stems from Daskalopoulous and Sesum \cite{DS06} where such solutions have been constructed for the logarithmic diffusion equation in $\real^2$ (which in differential geometry is a particular case of the Ricci flow), although they have been noticed previously for Eq. \eqref{FDE} with exactly $m=m_c$ in \cite{GPV00}. More precisely, we are looking for self-similar solutions of the following form
\begin{equation}\label{expSS}
u(x,t)=e^{\alpha t}f(|x|e^{-\beta t}), \qquad \alpha, \ \beta\in\real,
\end{equation}
for some suitable profile $f$ solving the ordinary differential equation
\begin{equation}\label{ODE}
(f^m)''(\xi)+\frac{N-1}{\xi}(f^m)'(\xi)-\alpha f(\xi)+\beta\xi f'(\xi)+\xi^{\sigma}f(\xi)^p=0.
\end{equation}
Existence of such self-similar eternal solutions proved to be a very seldom phenomenon, due to the fact that parabolic equations typically enjoy smoothing effects not allowing for the time variable to move independently in both directions; however, they appear sometimes in critical cases of exponents that split the general dynamics of the equations into two regimes, such as for example $m=m_c$ for Eq. \eqref{FDE} or $p=p_c=2N/(N+1)$ for the parabolic $p$-Laplacian equation \cite{ISV08} and more recently Lauren\ced{c}ot and one of the authors constructed eternal solutions for a fast diffusion equation involving gradient absorption \cite{IL13}.

Going back to our Eq. \eqref{eq1}, we are able to obtain eternal solutions exactly for the critical value of $\sigma$ in \eqref{sig}. We state our main result below.
\begin{theorem}\label{th.1}
Let $N\geq3$, $m\in(0,m_c)$, $p>1$ and $\sigma$ defined in \eqref{sig}. Then there exist \textbf{unique exponents} $\alpha$ and $\beta\in\real$ such that
\begin{equation}\label{relation}
\alpha=\frac{2}{m-1}\beta
\end{equation}
and a \textbf{unique profile} $f$ such that
$$
U(x,t)=e^{\alpha t}f(\xi), \qquad \xi=|x|e^{-\beta t}
$$
is a solution to Eq. \eqref{eq1} in the sense that the profile $f(\xi)$ solves the ordinary differential equation \eqref{ODE}. The profile $f$ satisfies
\begin{equation}\label{profile.prop}
f(0)=A>0, \ f'(0)=0, \qquad f(\xi)\sim C\xi^{(2-N)/m}, \ {\rm as} \ \xi\to\infty,
\end{equation}
and the \textbf{signs} of the self-similarity exponents \textbf{change exactly at $m=m_s$}, independent of $p$: $\alpha>0$ and $\beta<0$ for $m\in(0,m_s)$, respectively $\alpha<0$ and $\beta>0$ for $m\in(m_s,m_c)$. Moreover, $U(t)\in L^1(\real^N)$ for any $t>0$, and these solutions cease to exist for $m\in[m_c,1)$.
\end{theorem}
Let us notice here that the behavior of the profile as $\xi\to\infty$ in \eqref{profile.prop} is the same as the decay \eqref{decay.anom} of the anomalous solutions for the fast diffusion equation \eqref{FDE}, but the dynamics is totally different: the solutions in Theorem \ref{th.1} are global and eternal, they do not either extinguish or blow up in finite time, as it is obvious from their form. This shows that none of the two terms is dominant over the other one in Eq. \eqref{eq1}.

\medskip

\noindent \textbf{Consequences of the change of signs at $m=m_s$}. A very interesting feature is the \emph{change of sign of both exponents} at $m=m_s$. This is a significant difference with respect to the anomalous solutions \eqref{anomalous} to the standard fast diffusion equation, where only exponent $\beta(m)$ changes sign at $m=m_s$, but $\alpha(m)>0$ for any $m\in(0,m_c)$. This change of sign also brings some striking consequences for the qualitative behavior of the solutions, in various aspects

$\bullet$ \emph{Rate of the exponential decay as $t\to\infty$}. On the one hand for $m\in(m_c,m_s)$ we have $\alpha<0$ and $\beta>0$, thus for a fixed $|x|=r\in(0,\infty)$ we have $\xi=|x|e^{-\beta t}\to0$ as $t\to\infty$, whence we get the asymptotic exponential decay
\begin{equation}\label{behmc}
U(x,t)\sim e^{\alpha t}f(0)=Ae^{\alpha t}, \qquad {\rm as} \ t\to\infty.
\end{equation}
On the other hand, for $m\in(0,m_s)$ we have $\alpha>0$, $\beta<0$, hence $\xi=|x|e^{-\beta t}\to0$ as $t\to\infty$ and, taking into account the tail of the profiles as $\xi\to\infty$ we get for any $x\neq0$ that
\begin{equation}\label{beh0}
\begin{split}
U(x,t)\sim Ce^{\alpha t}(|x|e^{-\beta t})^{(2-N)/m}&=C|x|^{(2-N)/m}\exp\left(\alpha+\frac{N-2}{m}\beta t\right)\\&=C|x|^{(2-N)/m}\exp\left(\frac{N(m-m_c)\beta}{m(m-1)}t\right)
\end{split}
\end{equation}
as $t\to\infty$, which is again decreasing in $t$ since the coefficient inside the last exponential is negative. For fixed $x\in\real^N\setminus\{0\}$, in both cases the solutions decay exponentially as $t\to\infty$ but with different rates.

$\bullet$ \emph{Form of the profiles}: by (at least formally) evaluating \eqref{ODE} at $\xi=0$ we deduce that the profiles have a minimum point at the origin if $\alpha>0$ and a maximum point at the origin if $\alpha<0$. Moreover, if $\alpha<0$ we cannot have minima at any $\xi_0>0$: at such a minimum point we would get from \eqref{ODE} that
$$
(f^m)''(\xi_0)-\alpha f(\xi_0)+\xi_0^{\sigma}f(\xi_0)^p=0
$$
and a contradiction since all the terms above are positive, which proves that the profiles are decreasing for $\xi\in(0,\infty)$. This leads to the following \emph{striking difference} between the geometry of the solutions given by \eqref{expSS}: for $\alpha>0$, that is $m\in(0,m_s)$, the maximum point of the solution moves towards $x=0$ as $t\to\infty$, leading to the formation of a \emph{boundary layer} near the origin, while for $\alpha<0$, that is $m\in(m_s,m_c)$, solutions are just decreasing at any $t>0$.

$\bullet$ \emph{Evolution of the mass in opposite way}. Since $U(t)\in L^1(\real^N)$ for any $t>0$, we can define the mass of the solution at time $t$ by $M(t)=\|U(t)\|_{L^1(\real^N)}$. We can thus relate this mass to the integral of the profile by the following calculation based on an obvious change of variable:
$$
M(t)=\int_{\real^N}e^{\alpha t}f(|x|e^{-\beta t})\,dx=e^{(\alpha+N\beta)t}\int_{\real^N}f(\xi)\,d\xi,
$$
and we notice that
$$
\alpha+N\beta=\left(1+\frac{N(m-1)}{2}\right)\alpha=\frac{N(m-m_c)}{2}\alpha,
$$
which is positive for $\alpha<0$ (that is, $m\in(m_s,m_c)$) and negative for $\alpha>0$ (that is, $m\in(0,m_s)$). We infer that for $m\in(0,m_s)$ solutions $U$ in Theorem \ref{th.1} \emph{lose mass} as $t>0$ increases (an effect showing that the fast diffusion is a bit stronger in this range) while for $m\in(m_s,m_c)$ solutions $U$ in Theorem \ref{th.1} \emph{gain mass} as $t>0$ increases (an effect showing that the reaction is a bit stronger in this range).

$\bullet$ \emph{Explicit stationary solutions at $m=m_s$}. All the previous analysis shows that $m=m_s$ is a kind of bifurcation between two regimes with different properties, and in the middle, exactly at $m=m_s$ (and any $p>1$) we have \emph{stationary solutions} that will be made explicit in Subsection \ref{subsec.stat}. For now, we plot in Figure \ref{fig1} the evolution of two self-similar eternal solutions with respect to time, one from each range $\alpha>0$ and $\alpha<0$, showing the contrast between their properties as explained above.

\begin{figure}[ht!]
  \begin{center}
  \subfigure[$\alpha>0$]{\includegraphics[width=7.5cm,height=6cm]{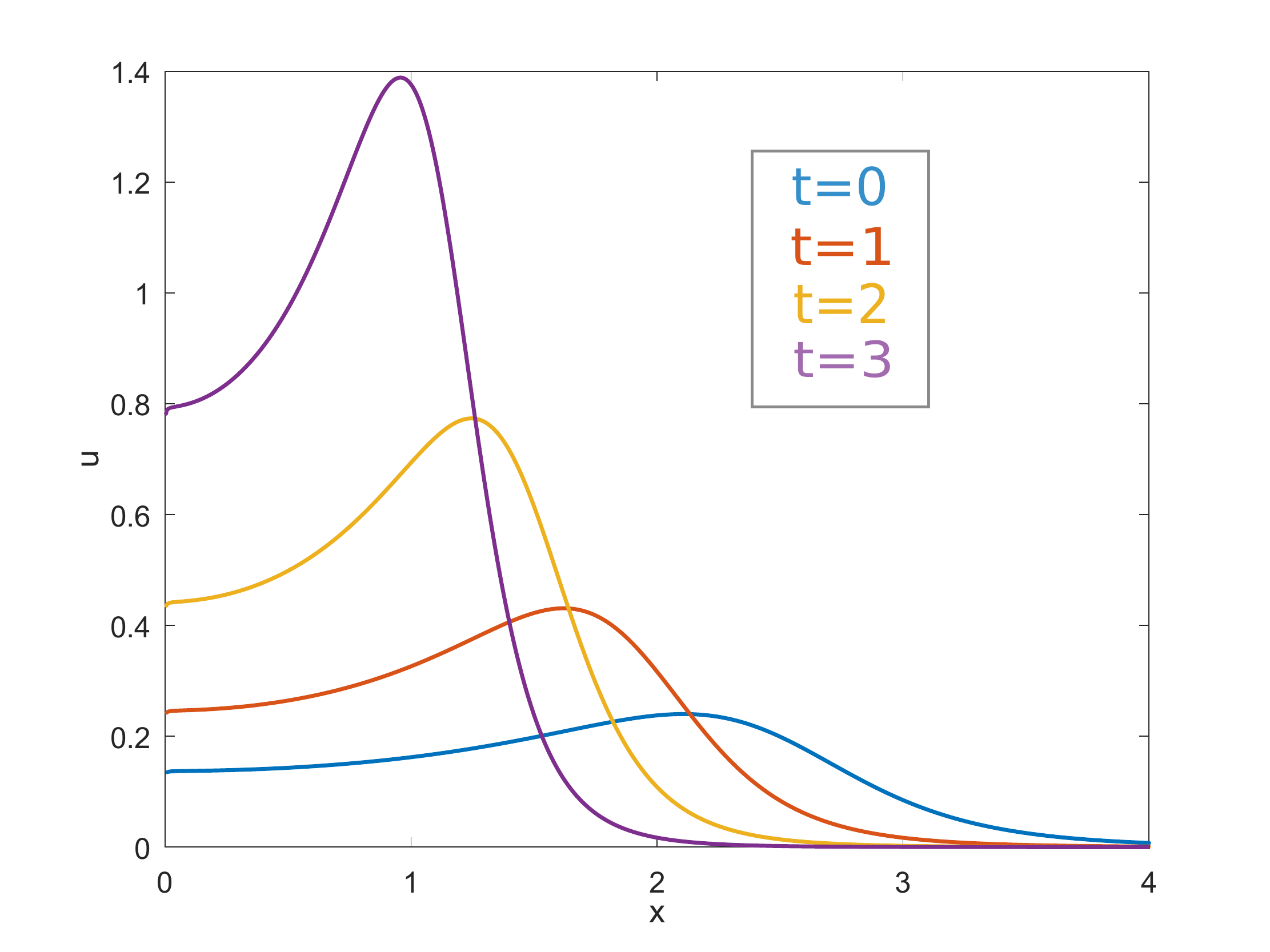}}
  \subfigure[$\alpha<0$]{\includegraphics[width=7.5cm,height=6cm]{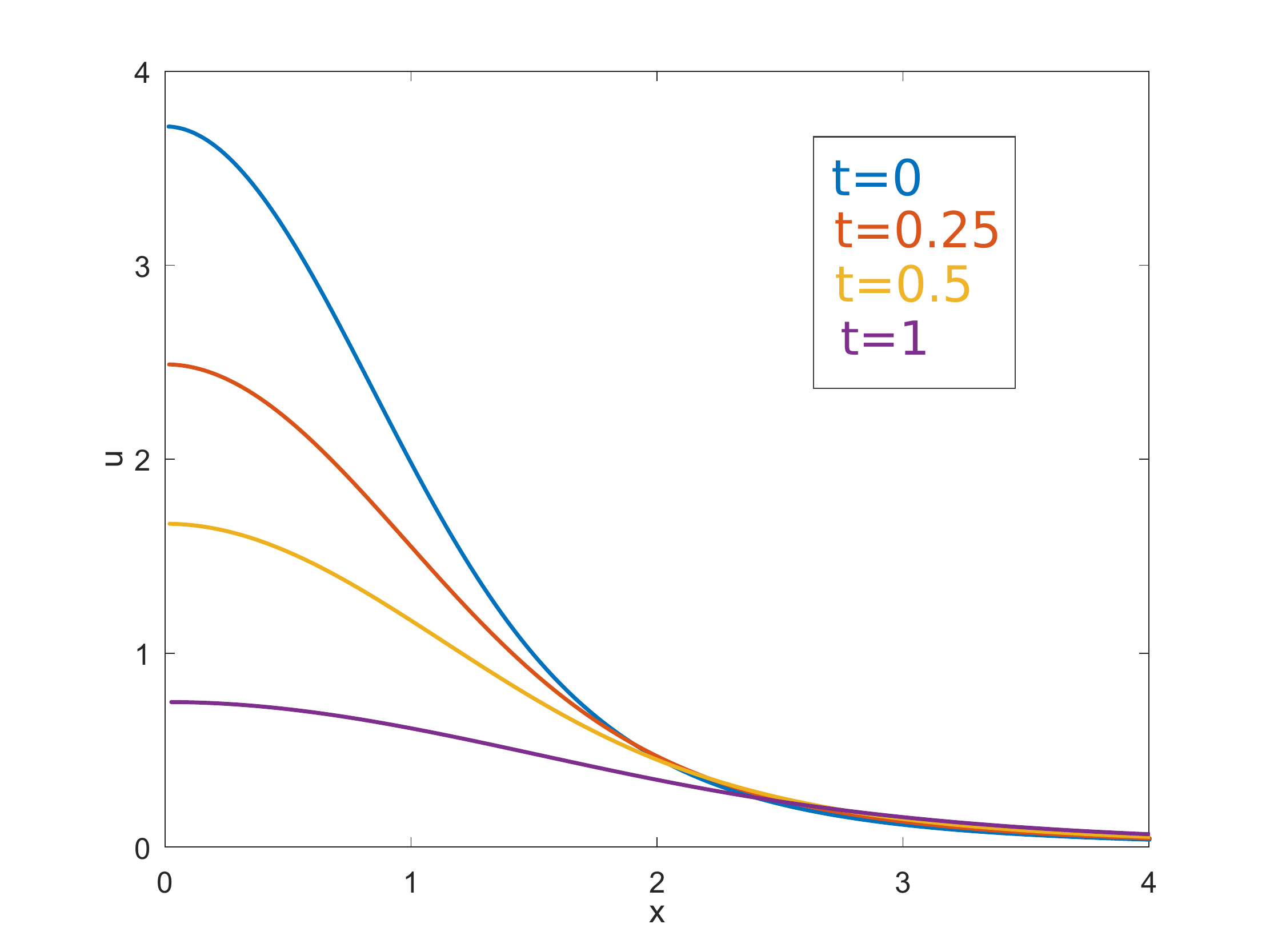}}
  \end{center}
  \caption{Eternal solutions at different times. Experiment for $N=3$, $p=2$ and $m=0.1$, respectively $m=4/15$}\label{fig1}
\end{figure}

In a recent paper \cite{IS21b} the authors proved existence and uniqueness of eternal solutions to Eq. \eqref{eq1} but with $m>1$ and $p<1$ and the same value \eqref{sig} of $\sigma$. Despite the fact that the equation is algebraically the same, the present problem is qualitatively very different: in the range $m>1$ we have slow diffusion instead of fast, and the fact that $m>p$ in \cite{IS21b} leads to compactly supported profiles (finite speed of propagation), while in the present work we deal with infinite speed of propagation and thus profiles with tails as $\xi\to\infty$. Despite these differences, in some technical parts of some of the proofs we will borrow analysis done in \cite{IS21b} in order to shorten the presentation. Let us stress here again that the eternal solutions obtained in \cite{IS21b} have always $\alpha>0$ and $\beta>0$.

\medskip

\noindent \textbf{Structure of the paper}. We divide the present work into three sections, apart from the Introduction. In Section \ref{sec.local} we construct an autonomous dynamical system associated to Eq. \eqref{ODE} and study it locally, in a neighborhood of each of its critical points. The global analysis of the phase plane is the subject of the longer Section \ref{sec.global}, which is at its turn divided into several subsections and will contain the proof of Theorem \ref{th.1}. Finally, in Section \ref{sec.appendix} we gather several particular facts in order to complete the presentation: a number of explicit or semi-explicit solutions to Eq. \eqref{ODE} and a self-map between radially symmetric solutions to Eq. \eqref{eq1}, generalizing a self-map for the fast diffusion equation \eqref{FDE}.

\section{The phase plane. Local analysis}\label{sec.local}

In this section we transform Eq. \eqref{ODE} into an autonomous dynamical system and analyze its behavior near the critical points. To this end, we have to fix one of the two possible signs for the exponent $\alpha$, and we will work with $\alpha>0$ (and thus $\beta<0$) in order to use the similarity in some technical steps with the proofs in \cite{IS21b}. With this convention, we consider the same change of variable as in \cite{IS21b} (inspired in fact by the one used in \cite{ISV08}) by letting
\begin{equation}\label{change}
X=\frac{\alpha}{2m}\xi^2f(\xi)^{1-m}, \qquad Y=\xi f'(\xi)f^{-1}(\xi), \qquad \frac{d}{d\eta}=\xi\frac{d}{d\xi},
\end{equation}
leading after straightforward calculations to the following autonomous dynamical system
\begin{equation}\label{PPSyst}
\left\{\begin{array}{ll}\dot{X}=X(2+(1-m)Y),\\ \dot{Y}=-mY^2-(N-2)Y+2X+(1-m)XY-KX^{(p-m)/(1-m)},\end{array}\right.
\end{equation}
where the derivative is taken with respect to the new independent variable $\eta$ introduced in \eqref{change} and which varies depending on the parameter
\begin{equation}\label{param}
K=\frac{1}{m}\left(\frac{2m}{\alpha}\right)^{(p-m)/(1-m)}.
\end{equation}
We notice that $X\geq0$, that the line $X=0$ is invariant for the system \eqref{PPSyst} and $Y$ might take any real value. The critical points in the finite part of the phase plane are
$$
P_0=(0,0), \qquad P_1=\left(0,-\frac{N-2}{m}\right), \qquad P_2=(X(P_2),Y(P_2)),
$$
where the critical point $P_2$ only exists for $m\in(0,m_c)$ and has the coordinates
\begin{equation}\label{coordP2}
X(P_2)=\left[\frac{2N(m_c-m)}{K(1-m)^2}\right]^{(1-m)/(p-m)}, \qquad Y(P_2)=-\frac{2}{1-m}.
\end{equation}
From now on, we restrict ourselves to the subcritical range $m\in(0,m_c)$ and perform the local analysis of the system near these points.

\subsection{Local analysis of the finite critical points}\label{subsec.finite}

As we shall see in the sequel, the three finite critical points are the most important ones for the analysis. We study them one by one below.
\begin{lemma}[Local analysis near $P_0$]\label{lem.P0}
The critical point $P_0$ is a saddle point. There exists a unique orbit going out of it into the phase plane, and the profiles contained in it satisfy $f(0)=A>0$, $f'(0)=0$.
\end{lemma}
\begin{proof}
The proof is completely identical to the proof of Lemma 2.1 in \cite{IS21b}. We give a sketch for the sake of completeness. It is straightforward to check that the linearization of the system \eqref{PPSyst} in a neighborhood of $P_0$ has a matrix with eigenvalues $\lambda_1=2$, $\lambda_2=2-N$ and corresponding eigenvectors $e_1=(1,2/N)$, $e_2=(0,1)$, being thus a saddle point. We are interested in the orbit going out of $P_0$ tangent to $e_1$, which contains profiles satisfying $X/Y\sim N/2$. By replacing $X$, $Y$ by their expressions in \eqref{change} and integrating, we readily get the claimed behavior of the profiles.
\end{proof}
For the critical point $P_1$ we already notice an important difference with respect to the analysis performed in \cite{IS21b}.
\begin{lemma}[Local analysis near $P_1$]\label{lem.P1}
The critical point $P_1$ is also a saddle point for $m\in(0,m_c)$. There exists a unique orbit entering it and coming from the positive part of the phase plane. This orbit contains profiles such that
\begin{equation}\label{beh.P1}
f(\xi)\sim C\xi^{-(N-2)/m}, \qquad {\rm as} \ \xi\to\infty, \ \ C>0 \ {\rm free \ constant}.
\end{equation}
\end{lemma}
\begin{proof}
The linearization of the system \eqref{PPSyst} near $P_1$ has the matrix
$$M(P_1)=\left(
  \begin{array}{cc}
    \frac{mN-N+2}{m} & 0 \\
    \frac{mN-N+2}{m} & N-2 \\
  \end{array}
\right),$$
with eigenvalues $\lambda_1=(mN-N+2)/m<0$, $\lambda_2=N-2>0$, thus is a saddle point for $m<m_c$. The unique orbit entering $P_1$ has $X\to0$, $Y\sim-(N-2)/m$, and the latter implies
$$
\frac{f'(\xi)}{f(\xi)}\sim-\frac{N-2}{m\xi},
$$
which gives \eqref{beh.P1} by direct integration. We recall now the definition of $X$ in \eqref{change} to infer that on this orbit entering $P_1$ we have
$$
X(\xi)\sim C\xi^{(mN-N+2)/m}, \qquad X(\xi)\to0,
$$
which, together with the fact that $mN-N+2=N(m-m_c)<0$, shows that the limit in the local behavior is taken as $\xi\to\infty$ and the proof is complete.
\end{proof}
We are finally left with the most interesting critical point, which is completely new with respect to the case $m>1$, $p<1$ analyzed in \cite{IS21b} and whose analysis is more involved. Recall here the Sobolev exponent $m_s$ defined in \eqref{ms}.
\begin{lemma}[Local analysis near $P_2$]\label{lem.P2}
The type of the critical point $P_2$ depends on the value of $m$ as follows:

(i) For $m\in[m_s,m_c)$, the critical point $P_2$ is either an unstable node or an unstable focus, depending on the value of the parameter $K\in(0,\infty)$ defined in \eqref{param}.

(ii) For $m\in(0,m_s)$, the critical point $P_2$ can be: an unstable node, an unstable focus, a center, a stable focus and a stable node, all them in dependence of the parameter $K\in(0,\infty)$.

In both cases, the orbits going out of $P_2$ (respectively entering $P_2$) contain profiles having the following limit behavior
\begin{equation}\label{beh.P2}
f(\xi)\sim\left[\frac{2m}{\alpha}X(P_2)\right]^{1/(1-m)}\xi^{-2/(1-m)},
\end{equation}
taken as $\xi\to0$ if the orbit goes out (profiles with a vertical asymptote at $\xi=0$) or as $\xi\to\infty$ if the orbit enters $P_2$ (profiles with a different tail at infinity).
\end{lemma}
\begin{proof}
The linearization of the system \eqref{PPSyst} near $P_2$ has the matrix
$$M(P_2)=\left(
  \begin{array}{cc}
    0 & (1-m)X(P_2) \\
    -\frac{K(p-m)}{1-m}X(P_2)^{(p-1)/(1-m)} & \frac{(N+2)(m-m_s)}{1-m}+(1-m)X(P_2) \\
  \end{array}
\right),$$
where we recall that $X(P_2)$ depends on the parameter $K$ and is defined in \eqref{coordP2}. Letting
$$
L:=(m-1)^2X(P_2)+(N+2)(m-m_s),
$$
we find that the eigenvalues of the matrix $M(P_2)$ are
\begin{equation}\label{interm1}
\lambda_{1,2}=\frac{1}{2(1-m)}\left[L\pm\sqrt{L^2-8N(p-m)(m_c-m)}\right].
\end{equation}
Let us consider first $m\in[m_s,m_c)$. In this case, taking into account \eqref{coordP2}, we find that $L>0$ for any $K>0$, thus the two eigenvalues in \eqref{interm1} are either real and positive or complex conjugated with positive real part (equal to $L$). It follows that the critical point $P_2$ is always unstable: either an unstable node or an unstable focus. Going now to the range $m\in(0,m_s)$, we notice that $L$ might change sign in dependence on the parameter $K$ and this introduces a big difference in the analysis. Indeed, noticing in \eqref{coordP2} that $X(P_2)$ depends on $K$ in a decreasing way, with $X(P_2)\to\infty$ as $K\to0$ and $X(P_2)\to0$ as $K\to\infty$, we get that

$\bullet$ for $K>0$ sufficiently small, $X(P_2)$ is big and $L>0$, in fact $L\to\infty$ as $K\to0$, thus the eigenvalues in \eqref{interm1} are both real positive numbers. We get an \emph{unstable node}.

$\bullet$ there exists a value of $K$ for which $L=L^*=\sqrt{8N(p-m)(m_c-m)}>0$. Above this value of $K$ the critical point $P_2$ becomes an \emph{unstable focus}, as $\lambda_1$, $\lambda_2$ become complex numbers with positive real parts.

$\bullet$ there exists a value of $K$, call it $K^*$, for which $L=0$, hence $\lambda_{1,2}$ become purely imaginary complex numbers. This means that the critical point $P_2$ can be \emph{either a center or a focus} for this precise instance of $K$.

$\bullet$ for $K>K^*$ but sufficiently close to $K^*$, $X(P_2)$ continues to decrease as $K$ increases and we have $L<0$, thus for such values of $K$ the eigenvalues in \eqref{interm1} are complex with negative real part, which means that $P_2$ is a \emph{stable focus}.

$\bullet$ finally, in some cases it is possible that $L=-L^*$ for some value of $K$, and if this happens, for higher values of $K$ we obtain $L^2>8N(p-m)(m_c-m)$ and we get two real, negative eigenvalues in \eqref{interm1}. In this case $P_2$ is a \emph{stable node}. This final case is possible only if there exists some value of $K\in(0,\infty)$ such that
$$
L=-L^*=-\sqrt{8N(p-m)(m_c-m)}>(N+2)(m-m_s),
$$
which after taking squares and performing easy calculations leads to the following condition on $m$, $N$ and $p$
$$
(m-1)^2(N-2)^2-8Np(m_c-m)\geq0.
$$
However, in the subsequent analysis the difference between nodes and foci will not be relevant. Finally, the orbits either entering or going out of $P_2$ have $X\to X(P_2)$ and $Y\to Y(P_2)$, the former of these together with the expression of $X$ in \eqref{change} leading directly to the local behavior \eqref{beh.P2}. Such behavior can be taken either as $\xi\to0$ (on orbits going out of $P_2$) or as $\xi\to\infty$ (on orbits entering $P_2$), but the intermediate case of a limit $\xi\to\xi_0$ for some $\xi_0\in(0,\infty)$ is discarded easily by a contradiction with the fact that $Y(\xi)\to-2/(1-m)$.
\end{proof}
This change of the character of $P_2$ from an unstable point into a stable point will be the decisive feature allowing for the existence of good orbits and thus profiles with  behavior as in \eqref{profile.prop} for $m\in(0,m_s)$ (in our framework with exponent $\alpha>0$), while the fact that this point does not change for $m\in(m_s,m_c)$ will be an obstacle for existence.

\subsection{Critical points at infinity}\label{subsec.inf}

This analysis follows closely the corresponding one performed for the case $m>1$ and $p<1$ in \cite{IS21b}, thus at some points we will skip some technical steps and refer to this previous work. The sign of $m+p-2$ makes a difference, as it follows below. We pass to the Poincar\'e sphere by following the theory in \cite[Section 3.10]{Pe} and introducing the new variables $(\overline{X},\overline{Y},W)$ such that
$$
X=\frac{\overline{X}}{W}, \qquad Y=\frac{\overline{Y}}{W}
$$
and recall that the critical points at infinity of the system \eqref{PPSyst} lie on the equator of the sphere, that is, they are points $(\overline{X},\overline{Y},0)$ with $\overline{X}^2+\overline{Y}^2=1$. Let now $P(X,Y)$, $Q(X,Y)$ be the right-hand sides of the two equations of the system \eqref{PPSyst}. The difference with respect to the sign of $m+p-2$ leads to the following three cases:

\medskip

\noindent $\bullet$ if $m+p>2$, that is $(p-m)/(1-m)>2$, then the highest order term in the expressions of $P(X,Y)$, $Q(X,Y)$ is $X^{(p-m)/(1-m)}$. We can thus let
$$
P^*(\overline{X},\overline{Y},W)=W^{(p-m)/(1-m)}P\left(\frac{\overline{X}}{W},\frac{\overline{Y}}{W}\right), \ \  Q^*(\overline{X},\overline{Y},W)=W^{(p-m)/(1-m)}Q\left(\frac{\overline{X}}{W},\frac{\overline{Y}}{W}\right)
$$
and follow the theory in \cite[Section 3.10]{Pe} to get that the critical points at infinity are given by the zeros of the expression obtained by letting $W=0$ in the following calculation
$$
[\overline{X}Q^*(\overline{X},\overline{Y},W)-\overline{Y}P^*(\overline{X},\overline{Y},W)]\Big|_{W=0}=-K\overline{X}^{(p+1-2m)/(1-m)},
$$
where the detailed calculations are given in \cite[Section 5]{IS21b}. We thus get two critical points at infinity that on the Poincar\'e sphere have coordinates $Q_2=(0,1,0)$, $Q_3=(0,-1,0)$.

\medskip

\noindent $\bullet$ if $m+p<2$, that is $(p-m)/(1-m)<2$, then the highest order terms in the expressions of $P(X,Y)$, $Q(X,Y)$ are all quadratic. We can thus set
\begin{equation}\label{interm5}
P^*(\overline{X},\overline{Y},W)=W^{2}P\left(\frac{\overline{X}}{W},\frac{\overline{Y}}{W}\right), \ \  Q^*(\overline{X},\overline{Y},W)=W^{2}Q\left(\frac{\overline{X}}{W},\frac{\overline{Y}}{W}\right)
\end{equation}
and follow the theory in \cite[Section 3.10]{Pe} to get that the critical points at infinity are given by the zeros of the expression obtained by letting $W=0$ in the following calculation
$$
[\overline{X}Q^*(\overline{X},\overline{Y},W)-\overline{Y}P^*(\overline{X},\overline{Y},W)]\Big|_{W=0}=-\overline{X}\overline{Y}(\overline{Y}-(1-m)\overline{X}),
$$
where the detailed calculations are given in \cite[Section 2.2]{IS21b}. We thus obtain, apart from the critical points $Q_2$, $Q_3$ identified in the previous case, two new critical points at infinity
\begin{equation}\label{crit.infty}
Q_1=(1,0,0), \qquad Q_4=\left(\frac{1}{\sqrt{1+(1-m)^2}},\frac{1-m}{\sqrt{1+(1-m)^2}},0\right).
\end{equation}

\medskip

\noindent $\bullet$ if $m+p=2$, that is $(p-m)/(1-m)=2$, we set again \eqref{interm5}, but in this case one more term contributes to the critical points of infinity, since we have
$$
[\overline{X}Q^*(\overline{X},\overline{Y},W)-\overline{Y}P^*(\overline{X},\overline{Y},W)]\Big|_{W=0}=-\overline{X}\left[\overline{Y}^2-(1-m)\overline{X}\overline{Y}+K\overline{X}^2\right],
$$
where the detailed calculations are given in \cite[Section 3.1]{IS21b}. Checking for the zeros of the right-hand side of the previous calculation and letting $\overline{Y}=\lambda\overline{X}$, we obtain, apart from the critical points $Q_2$ and $Q_3$, two more critical points
\begin{equation}\label{crit.infty.bis}
Q_1=\left(\frac{1}{\sqrt{1+y_1^2}},\frac{y_1}{\sqrt{1+y_1^2}},0\right), \qquad Q_4=\left(\frac{1}{\sqrt{1+y_2^2}},\frac{y_2}{\sqrt{1+y_2^2}},0\right),
\end{equation}
where
\begin{equation}\label{y12}
y_{1,2}=\frac{(1-m)\pm\sqrt{(m-1)^2-4K}}{2},
\end{equation}
are obtained as the roots of the equation $\lambda^2+(m-1)\lambda+K=0$, provided $K\leq(m-1)^2/4$. All the details are given in \cite[Section 3.1]{IS21b}.

We perform the local analysis of the system \eqref{PPSyst} in a neighborhood of these points below.
\begin{lemma}[Local analysis near $Q_2$ and $Q_3$]\label{lem.Q23}
The critical point $Q_2=(0,1,0)$ is an unstable node and the critical point $Q_3=(0,-1,0)$ is a stable node. The orbits either going out of $Q_2$ or entering $Q_3$ contain profiles having a change of sign at some finite point $\xi_0\in[0,\infty)$ with the local behavior near $Q_2$,
$$
f(\xi)\sim C(\xi-\xi_0)^{1/m}, \qquad {\rm as} \ \xi\to\xi_0, \ \xi>\xi_0,
$$
respectively the local behavior near $Q_3$
$$
f(\xi)\sim C(\xi_0-\xi)^{1/m}, \qquad {\rm as} \ \xi\to\xi_0, \ \xi<\xi_0.
$$
\end{lemma}
\begin{proof}[Formal proof]
We give here a more formal proof, which allows us to understand how the local behavior comes out. We know that when approaching the points $Q_2$ and $Q_3$, by the definition of the coordinates on the Poincar\'e sphere, we have $Y\to\pm\infty$ and $Y/X\to\pm\infty$. Thus, we go to the system \eqref{PPSyst} and estimate the first order approximation of $dY/dX$ by neglecting the lower order terms (under the previous assumptions) and maintaining only the dominating (or possibly dominating) ones in $P(X,Y)$, $Q(X,Y)$ to get
$$
\frac{dY}{dX}\sim-\frac{mY^2+KX^{(p-m)/(1-m)}}{(1-m)XY},
$$
and we infer by integration that the trajectories of the system satisfy
\begin{equation}\label{interm2}
Y^2+\frac{2K}{m+p}X^{(p-m)/(1-m)}\sim CX^{-2m/(1-m)}, \qquad C\in\real \ {\rm free \ constant}
\end{equation}
in a neighborhood of the points $Q_2$ and $Q_3$. Since $Y^2\to\infty$, \eqref{interm2} forces $X\to0$ on the orbits when approaching $Q_2$ and $Q_3$ and we finally get the approximation
\begin{equation}\label{interm3}
Y\sim CX^{-m/(1-m)}, \qquad C\in\real \ {\rm free \ constant},
\end{equation}
where $C>0$ for the orbits going out of $Q_2$ and $C<0$ for the orbits entering $Q_3$. Putting \eqref{interm3} in terms of profiles by using the definitions of $X$ and $Y$ in \eqref{change}, we get
$$
(f^{m-1}f')(\xi)\sim C\xi^{-(m+1)/(1-m)},
$$
which by integration leads to
\begin{equation}\label{interm4}
f(\xi)\sim\left(C_1+C\xi^{-2m/(1-m)}\right)^{1/m}.
\end{equation}
It remains to show that the local behavior in \eqref{interm4} is taken as $\xi\to\xi_0\in(0,\infty)$ for $Q_3$ and $\xi\to\xi_0\in[0,\infty)$ for $Q_2$. This follows from the definition of $X$ and the fact that $X\to0$ in a neighborhood of $Q_2$ or $Q_3$, which means
$$
X(\xi)=\xi^2f(\xi)^{1-m}=\xi^2\left(C_1+C\xi^{-2m/(1-m)}\right)^{(1-m)/m}=\left(C+C_1\xi^{2m/(1-m)}\right)^{(1-m)/m}\to0,
$$
which does not allow taking a limit as $\xi\to\infty$. It is then obvious that the limit $\xi\to0$ is not allowed on the orbits entering $Q_3$, while it can be allowed along trajectories going out of $Q_2$. Finally, the behavior in \eqref{interm4} as $\xi\to\xi_0\in(0,\infty)$ is equivalent to the one in the statement of the Lemma, following an easy discussion on the signs of the constants that is given in detail at the end of \cite[Lemma 2.4]{IS21b}. A \emph{fully rigorous proof} can be done by using the theory in \cite[Theorem 2, Section 3.10]{Pe} in line with the proof of \cite[Lemma 2.4]{IS21b}.
\end{proof}
For the critical points $Q_1$ and $Q_4$ introduced in \eqref{crit.infty} or \eqref{crit.infty.bis}, which only exist if $m+p\leq2$, the situation is different with respect to \cite{IS21b}, but as we shall see, they are not very important for the subsequent analysis of the phase plane.
\begin{lemma}\label{lem.Q1Q4}
The critical point $Q_1$ on the Poincar\'e sphere is a saddle-node according to the theory in \cite[Section 2.11]{Pe} and there is a unique orbit entering this point and coming from the finite part of the phase plane. The critical point $Q_4$ is a stable node. The orbits entering both $Q_1$ and $Q_4$ and coming from the finite part of the phase plane contain profiles having a vertical asymptote at some finite point $\xi_0\in(0,\infty)$, in the sense $f(\xi)\to\infty$ as $\xi\to\xi_0$.
\end{lemma}
\begin{proof}
Let us restrict ourselves first to exponents such that $m+p<2$. The local analysis of both points can be performed, according to \cite[Theorem 2, Section 3.10]{Pe}, on the following system (also obtained in \cite[Lemma 2.3]{IS21b})
\begin{equation}\label{PPSyst2}
\left\{\begin{array}{ll}\dot{y}=2w^{(1-m)/(2-m-p)}+(1-m)y-Nyw^{(1-m)/(2-m-p)}-y^2-Kw, \\ \dot{w}=-(2-m-p)yw-\frac{2(2-m-p)}{1-m}w^{1+(1-m)/(2-m-p)},\end{array}\right.
\end{equation}
where $y=Y/X$, $z=1/X$ and $w=z^{(2-m-p)/(1-m)}$. More precisely, the critical point $Q_1$ is topologically equivalent to the critical point $(y,w)=(0,0)$ and the critical point $Q_4$ is topologically equivalent to the critical point $(y,w)=(1-m,0)$ in the system \eqref{PPSyst2}. The linearization of the system \eqref{PPSyst2} near the point $(y,w)=(1-m,0)$ equivalent to $Q_4$ has the matrix
$$M(Q_4)=\left(
  \begin{array}{cc}
    -(1-m) & -K \\
    0 & (1-m)(m+p-2)\\
  \end{array}
\right),$$
with two negative eigenvalues $\lambda_1=-(1-m)$ and $\lambda_2=(1-m)(m+p-2)$, thus $Q_4$ is a stable node. The orbits entering $Q_4$ are characterized by the fact that $Y/X\sim 1-m$ in a neighborhood of $Q_4$, which leads after an integration to
\begin{equation}\label{interm6}
f(\xi)\sim\left(C-\frac{(m-1)^2}{2}\xi^2\right)^{1/(m-1)}, \qquad C>0 \ {\rm free \ constant},
\end{equation}
presenting a vertical asymptote as $\xi\to\xi_0$ for some $\xi_0\in(0,\infty)$ (which can be made explicit in terms of the constant $C>0$) since $m-1<0$. The linearization of the system \eqref{PPSyst2} in a neighborhood of the origin has the matrix
$$M(Q_1)=\left(
  \begin{array}{cc}
    1-m & -K \\
    0 & 0 \\
  \end{array}
\right),$$
with eigenvalues $\lambda_1=1-m>0$ and $\lambda_2=0$. We thus have an unstable manifold and center manifolds (that may not be unique). The analysis of the center manifolds (following \cite[Section 2.12]{Pe}) show that their equation and direction of the flow over them are given by the following
$$
y(w)=\frac{K}{1-m}w+o(w), \qquad \dot{w}=-\frac{K(2-m-p)}{1-m}w^2+o(w^2),
$$
thus all the orbits tangent to some center manifold enter $Q_1$. We thus deduce from \cite[Theorem 1, Section 2.11]{Pe} that the critical point $Q_1$ is a saddle-node, where the "saddle sector" takes the orbits approaching $Q_1$ from the interior of the phase plane, while the "node sector" contains only orbits going out of $Q_1$ on the boundary of the Poincar\'e sphere. It thus follows that there is a unique orbit entering $Q_1$ from the interior of the phase plane, with
$$
y\sim\frac{K}{1-m}w=\frac{K}{1-m}z^{(2-m-p)/(1-m)}
$$
in a neighborhood of it, which writes equivalently
$$
Y\sim\frac{K}{1-m}X^{(p-1)/(1-m)}
$$
and in terms of profiles gives after substitution with the definitions in \eqref{change} and integration
$$
f(\xi)\sim\left(D-\frac{K}{2}\xi^{2(p-1)/(1-m)}\right)^{1/(1-p)}, \qquad D>0 \ {\rm free \ constant},
$$
which produces a vertical asymptote at some finite $\xi_0\in(0,\infty)$, since $1-p<0$. More details about the calculations are given in \cite[Lemma 2.3]{IS21b}.

For the remaining case $m+p=2$, where the critical points $Q_1$ and $Q_4$ are defined in \eqref{crit.infty.bis}, the analysis is very similar, since on orbits near both of them we have $Y/X\sim y_1$ or $Y/X\sim y_2$, with $y_1$, $y_2>0$ defined in \eqref{y12}, which readily lead to a similar vertical asymptotes as the one obtained in \eqref{interm6}. We omit here the details, that are similar to the ones in \cite[Section 3.1]{IS21b}.
\end{proof}

\noindent \textbf{Remark.} Profiles with a vertical asymptote at some finite positive point have been also noticed in the study of Eq. \eqref{eq1.hom} (that is, for $\sigma=0$) with $p>m>1$, see for example Figure 5.1 in \cite[p.214]{S4}.

\section{Global analysis. Proof of Theorem \ref{th.1}}\label{sec.global}

In this section we deduce how the trajectories go inside the phase plane associated to the system \eqref{PPSyst} and we prove Theorem \ref{th.1}. Let us notice that the result of Theorem \ref{th.1} is equivalent to the existence and uniqueness of a \emph{saddle-saddle connection} between $P_0$ and $P_1$. We will prove this next, and the main tool in the proof will be an argument of monotonicity. We divide the steps of the proof into several subsections.

\subsection{Orbits for $K>0$ small}\label{subsec.small}

The first preparatory step deals with the configuration of the phase plane for very small $K>0$ (that is, $X(P_2)$ very large).
\begin{lemma}\label{lem.small}
Let $N\geq3$, $m\in(0,m_c)$ and $p>1$ be fixed and let $\sigma$ as in \eqref{sig}. Then there exists $K_1>0$ (depending on $m$, $N$ and $p$) such that, for any $K\in(0,K_1)$, the unique orbit entering the saddle point $P_1$ in the phase plane associated to the system \eqref{PPSyst} comes from the critical point $P_2$.
\end{lemma}
\begin{proof}
Let us consider the line passing through $P_2$
$$
l: Y=X-X(P_2)+Y(P_2).
$$
The direction of the flow of the system \eqref{PPSyst} over the line $l$ is given by the sign of the expression
\begin{equation}\label{interm7}
\begin{split}
F(X)&=mX^2+\frac{2m(1-m)X(P_2)+(N+2)(m-m_s)}{m-1}X\\&+\frac{2N(m_c-m)}{(m-1)^2}\left[\frac{X}{X(P_2)}\right]^{(p-m)/(1-m)}+\frac{1}{(m-1)^2}A(X(P_2)),
\end{split}
\end{equation}
where
\begin{equation}\label{interm8}
A(X(P_2))=m(m-1)^2X(P_2)^2+(1-m)(N+1)(m-m_s)X(P_2)-2N(m_c-m)
\end{equation}
and we recall that $X(P_2)$ is defined (in terms of $K$) in \eqref{coordP2}. Noticing that $X(P_2)\to\infty$ as $K\to0$ and that both in the expression of $F(X)$ and of $A(X(P_2))$ we have $mX^2>0$ and the dominating terms with respect to $X(P_2)$ have positive coefficients in \eqref{interm7} and \eqref{interm8}, we readily get that $F(X)>0$ for any $X\geq0$ provided $X(P_2)$ sufficiently large, that is, $K>0$ sufficiently small. The intersection of the line $l$ with the $Y$ axis is reached at
$$
Y_0=Y(P_2)-X(P_2)<-\frac{N-2}{m}, \qquad {\rm for} \ X(P_2) \ {\rm large},
$$
hence the critical point $P_1$ lies on the same side as the origin with respect to the line $l$. The orbit entering $P_1$ cannot cross the line $l$ from right to left due to the fact that $F(X)>0$. On the other hand, considering the isocline $\dot{Y}=0$ of the system \eqref{PPSyst}, that is, the curve
$$
C: -mY^2-(N-2)Y+2X+(1-m)XY-KX^{(p-m)/(1-m)}=0,
$$
we notice that the curve $C$ connects $P_2$ and $P_1$ and splits the half-plane $\{Y<-2/(1-m)\}$ into two regions that we plot in Figure \ref{fig2} below, which gives a "visual proof" of this Lemma:

\medskip

$\bullet$ one region (I) enclosed by the curve $C$, the line $Y=-2/(1-m)=Y(P_2)$ and the $Y$ axis, in which $\dot{X}<0$, $\dot{Y}>0$, hence $dY/dX<0$ along the trajectories in this region.

$\bullet$ one region (II) lying below the line $Y=-2/(1-m)=Y(P_2)$ and in the exterior of the curve $C$, in which $\dot{X}<0$, $\dot{Y}<0$, hence $dY/dX>0$ along the trajectories in this region.

It readily follows from these signs of $dY/dX$ along the trajectories that the orbit entering $P_1$ comes through the second, exterior region. Moreover, the direction of the flow of the system \eqref{PPSyst} on the part of the curve $C$ which lies in the half-plane $\{Y<-2/(1-m)\}$ is given by the sign of the expression
$$
G(X,Y)=X[2+(1-m)Y]^2-\frac{K(p-m)}{1-m}[2+(1-m)Y]X^{(p-m)/(1-m)}>0,
$$
since $2+(1-m)Y<0$ in the half-plane where we work. Taking into account that the normal direction to the curve $C$ is given by
$$
\overline{n}(X,Y)=\left(2+(1-m)Y-\frac{K(p-m)}{1-m}X^{(p-1)/(1-m)},-2mY-(N-2)+(1-m)X\right)
$$
whose $X$-component is obviously negative in the region where $2+(1-m)Y<0$, it follows that no orbit can cross the part of the curve $C$ which lies in the half-plane $\{Y<-2/(1-m)\}$ from region (I) into region (II) above. We then conclude from this analysis that for such sufficiently small values of $K$ for which $F(X)>0$ in \eqref{interm7}, the orbit entering $P_1$ has to lie completely in the region limited by the part of the curve $C$ contained in the half-plane $\{Y<-2/(1-m)\}$, the line $l$ and the $Y$ axis, as shown in Figure \ref{fig2}. Since in this region (which is a part of region (II)) the components $X$ and $Y$ are monotonic along any trajectory, the orbit entering $P_1$ must come from the only critical point lying on the boundary of this region, which is $P_2$.
\end{proof}

\begin{figure}[ht!]
  \begin{center}
  \includegraphics[width=11cm,height=7.5cm]{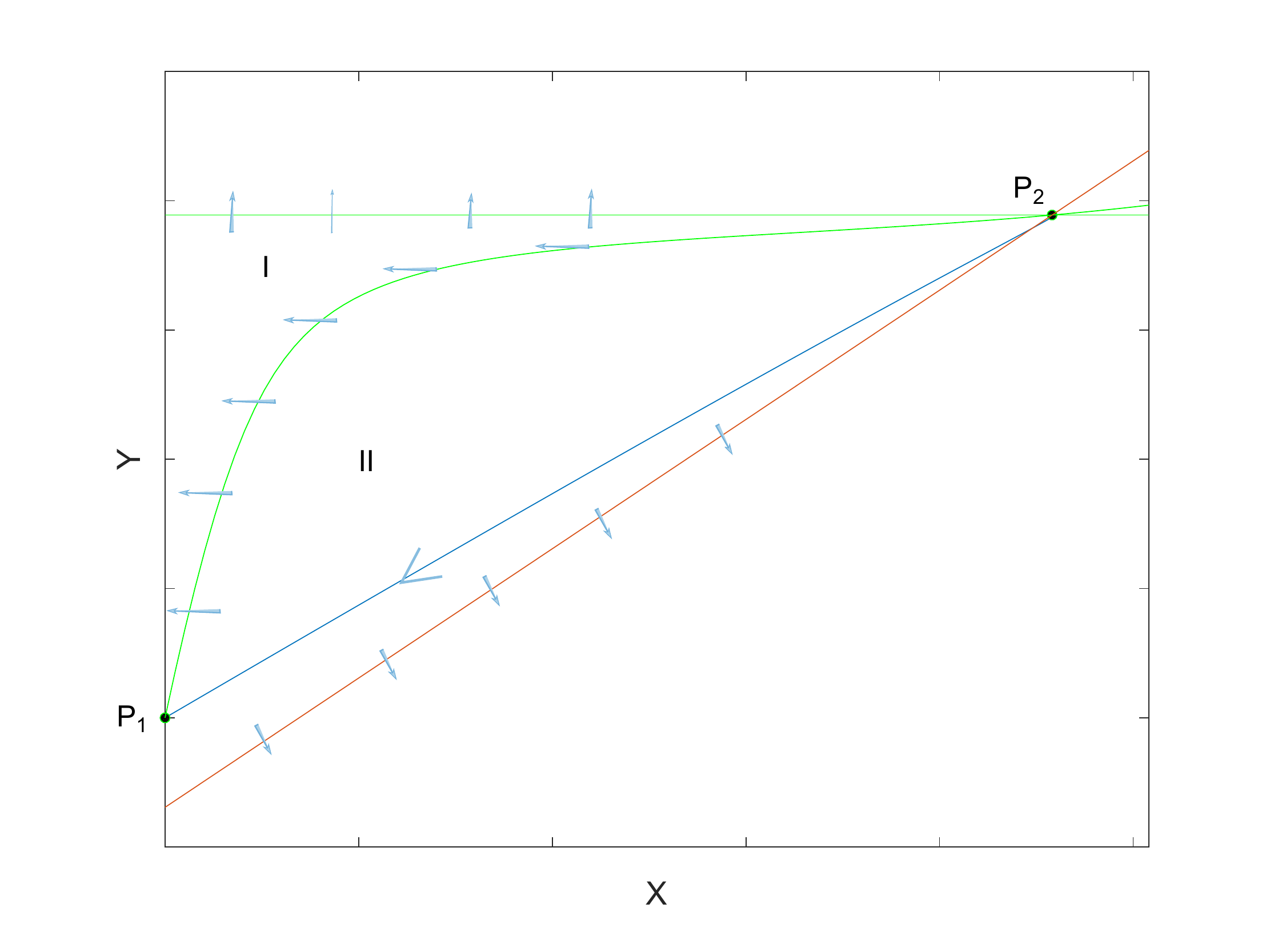}
  \end{center}
  \caption{The orbit connecting $P_2$ to $P_1$ for $K>0$ small}\label{fig2}
\end{figure}

\subsection{Monotonicity with respect to $K>0$}\label{subsec.monot}

For the easiness of the rest of the analysis, we perform a further change of variable that transforms the system \eqref{PPSyst} into a new one. Let us set
\begin{equation}\label{change2}
X=X(P_2)U, \qquad Y=\frac{cV-2}{1-m}, \qquad \eta=d\overline{\eta},
\end{equation}
where
$$
c=\sqrt{\frac{2N(m_c-m)}{m}}, \qquad d=\frac{1-m}{\sqrt{2mN(m_c-m)}}.
$$
We thus obtain a new system in variables $(U,V)$ introduced in \eqref{change2}
\begin{equation}\label{PPSystRen}
\left\{\begin{array}{ll}\dot{U}=C_mUV,\\ \dot{V}=-V^2-C_sV+1+C_{K}UV-U^{(p-m)/(1-m)},\end{array}\right.
\end{equation}
where the derivatives are taken with respect to the new independent variables $\overline{\eta}$ and
\begin{equation}\label{coefficients}
\begin{split}
&C_m=\frac{1-m}{m}, \qquad C_s=\frac{(N+2)(m_s-m)}{\sqrt{2mN(m_c-m)}},\\ &C_K=\frac{(1-m)^{2(p-1)/(p-m)}}{[2N(m_c-m)]^{(m+p-2)/2(p-m)}\sqrt{m}K^{(1-m)/(p-m)}}.
\end{split}
\end{equation}
Let us notice that in the system \eqref{PPSystRen} the critical points in the plane become
$$
P_0=\left(0,\frac{-C_s+\sqrt{C_s^2+4}}{2}\right), \ P_1=\left(0,\frac{-C_s-\sqrt{C_s^2+4}}{2}\right), \ P_2=(1,0)
$$
and an important feature of this system is the fact that $C_s$ changes sign (if moving $m$) at $m=m_s$, a fact that will become essential later. Let us introduce also the following notation: $l_0(K)$ be the (unique) orbit going out of the saddle point $P_0$ and $l_1(K)$ be the (unique) orbit entering the saddle point $P_1$. We are now interested in the change of the direction of the orbits $l_0(K)$ and $l_1(K)$ with respect to the parameter $K>0$. We have the following
\begin{lemma}[Monotonicity lemma]\label{lem.monot}
Let $K_1$, $K_2>0$ such that $K_1<K_2$. Then the orbit $l_0(K_1)$ "stays above" the orbit $l_0(K_2)$ in the half-plane $\{V>0\}$, before the first intersection with the $U$ axis, and the orbit $l_1(K_1)$ "stays above" the orbit $l_0(K_2)$ inside the half-plane $\{V<0\}$, after the last intersection with the $U$ axis. Here "stays above" means that, if for a fixed $U_0>0$ we let $V_1$, $V_2$ be the coordinates of the point on the orbits $l_0(K_1)$, $l_0(K_2)$ (respectively $l_1(K_1)$, $l_1(K_2)$) having $U=U_0$, then $V_1>V_2$ while $V_1>0$ on the orbits $l_0(K_i)$, $i=1,2$ (respectively $V_1>V_2$ while $V_2<0$ for the orbits $l_1(K_i)$, $i=1,2$).
\end{lemma}
\begin{proof}
We will redo the local analysis of the critical points $P_0$ and $P_1$ in our new system \eqref{PPSystRen} looking for the eigenvectors tangent to the orbits $l_0(K)$, $l_1(K)$ in a neighborhood of these points. In order to simplify the writing, let us consider a generic point $P=(0,b)$ for some $b\in\real$. The linearization of the system \eqref{PPSystRen} near the point $P=(0,b)$ has the matrix
$$
M(0,b)=\left(
         \begin{array}{cc}
           C_mb & 0 \\
           C_kb & -C_s-2b \\
         \end{array}
       \right),
$$
with eigenvalues and corresponding eigenvectors
\begin{equation}\label{interm9}
\lambda_1=C_mb, \ e_1=\left(\frac{C_sm+b(m+1)}{mbC_K},1\right), \qquad \lambda_2=-C_s-2b, \ e_2=(0,1).
\end{equation}
By particularizing $b$ as the $V$-coordinate of the critical points $P_0$, respectively $P_1$ and recalling the local analysis performed in Lemmas \ref{lem.P0} and \ref{lem.P1}, noticing that $b>0$ for $P_0$ and $b<0$ for $P_1$, we conclude that the orbits $l_0(K)$, respectively $l_1(K)$ go out of $P_0$, respectively enter $P_1$ tangent to the eigenvector $e_1$ in \eqref{interm9}. Moreover, in a neighborhood of these saddle points we get from the formulas of the $V$-components of them that $C_s=(1-b^2)/b$, thus
\begin{equation}\label{interm10}
mC_s+b(m+1)=\frac{m+b^2}{b}
\end{equation}
We now look at the dependence on $K$ of the orbits locally, in sufficiently small neighborhoods of $P_0$, respectively $P_1$. Taking into account \eqref{interm10} we have
\begin{equation*}
\frac{dV}{dU}\sim\frac{mbC_K}{mC_s+b(m+1)}=\frac{C(m,N,p)b^2m}{m+b^2}\frac{1}{K^{(p-m)/(1-m)}},
\end{equation*}
where $C(m,N,p)>0$ is the constant appearing in the formula of $C_K$ in \eqref{coefficients}. We thus infer the desired local monotonicity with respect to $K$ in a local neighborhood of the points. Moreover, along the trajectories we have
$$
\frac{dV}{dU}=\frac{C_KUV-U^{(p-m)/(1-m)}-V^2-C_sV+1}{C_mUV}=\frac{C_K}{C_m}-\frac{U^{(p-m)/(1-m)}+V^2+C_sV-1}{C_mUV}
$$
and this varies in a decreasing way with respect to $K>0$. By the comparison theorem, we infer that the orbits $l_0(K)$, $l_1(K)$ remain ordered for different values of $K$ at least while the sign of $V$ does not change, as claimed.
\end{proof}
The statement of Lemma \ref{lem.monot} cannot be extended further without any restrictions. Indeed, Lemma \ref{lem.small} shows that the orbits $l_1(K)$ with $K\in(0,K_1)$ all meet at the critical point $P_2$, despite being ordered before arriving (in the backward sense of their directions) to $P_2$. The next lemma shows that this is the only possible case of intersection over the $U$ axis
\begin{lemma}[Strict monotonicity outside $P_2$]\label{lem.strict}
Two orbits $l_1(K_1)$ and $l_1(K_2)$ with $K_1<K_2$ cannot intersect at points $(U,0)$ with $U>1$. The same result is valid also for two orbits $l_0(K_1)$ and $l_0(K_2)$.
\end{lemma}
\begin{proof}
Since $K_1<K_2$ it follows obviously from \eqref{coefficients} that $C_{K_1}>C_{K_2}$. Fix now $U>1$ and estimate the distance between the $V$-components of the two orbits (already ordered by Lemma \ref{lem.monot} before reaching $V=0$):
\begin{equation*}
\begin{split}
\frac{d(V_1-V_2)}{dU}&=\frac{C_{K_1}UV_1-U^{(p-m)/(1-m)}-V_1^2-C_sV_1+1}{C_mUV_1}\\&-\frac{C_{K_2}UV_2-U^{(p-m)/(1-m)}-V_2^2-C_sV_2+1}{C_mUV_2}\\
&=\frac{1}{C_m}\left[(C_{K_1}-C_{K_2})+\frac{1-U^{(p-m)/(1-m)}}{U}\left(\frac{1}{V_1}-\frac{1}{V_2}\right)\right]\\
&-\frac{1}{C_m}\frac{V_1-V_2}{U}.
\end{split}
\end{equation*}
Assume now for contradiction that two orbits, either $l_1(K_1)$ and $l_1(K_2)$, or $l_0(K_1)$ and $l_0(K_2)$, intersect at some point $(U,0)$ with $U>1$. Take then a small neighborhood of this intersection point $(U,0)$ still included in the half-plane $\{U>1\}$. Before the intersection, we know from Lemma \ref{lem.monot} that the orbits are still ordered and $V_1>V_2$ for the same value of $U$. Since $U>1$ in the neighborhood we have chosen, it follows that
$$
\frac{1-U^{(p-m)/(1-m)}}{U}\left(\frac{1}{V_1}-\frac{1}{V_2}\right)>0
$$
before the intersection of the two orbits. Moreover, since $K_1<K_2$ we have $C_{K_1}-C_{K_2}>0$. Letting $V_1$, $V_2$ tend to 0, we get from the previous equality that in the limit
$$
\frac{d(V_1-V_2)}{dU}\to\frac{C_{K_1}-C_{K_2}}{C_m}>0
$$
and a contradiction, as the latter says that the distance between the orbits increases instead of tending to zero (as it should happen at an intersection point). This argument is valid for both families of orbits $l_0$ and $l_1$.
\end{proof}
Noticing further that the direction of the flow of the system \eqref{PPSystRen} over the $U$ axis is given by the sign of $1-U^{(p-m)/(1-m)}$, we readily infer that no orbit can cross the $U$ axis from $\{V>0\}$ to $\{V<0\}$ at some point $(U,0)$ with $U<1$. This remark together with Lemmas \ref{lem.monot} and \ref{lem.strict} allow us to introduce the following functions of $K$. Let $U_0(K)$ be the coordinate of the first intersection point of the orbit $l_0(K)$ with the $U$ axis (after going out of $P_0$) and $U_1(K)$ be the last intersection point of the orbit $l_1(K)$ (before entering $P_1$), with the convention that $U_0(K)=+\infty$ if $l_0(K)$ enters one of the critical points $Q_1$ or $Q_4$ at infinity without crossing the $U$ axis. We have just proved that, if $K_1>0$ is the highest parameter for which the orbit $l_1(K)$ comes directly from $P_2$, then $U_1(K)$ is a strictly increasing function for $K>K_1$, and $U_0(K)$ is a strictly decreasing function when $1<U_0(K)<\infty$. Moreover, both functions are continuous with respect to $K$ (when taking finite values), as it follows from the continuity with respect to the parameter.

\subsection{Orbits for $K>0$ large and final argument when $m\in(0,m_s)$}\label{subsec.large}

The next step is to study the other extremal configuration of the phase plane, for $K>0$ very large. Let us restrict ourselves for this study to the range $m\in(0,m_s)$, which makes an important difference with respect to other ranges of $m$, as it follows from Lemma \ref{lem.P2}. Indeed, the critical point $P_2$ is the one that drives the whole picture of the phase plane, and the fact that it changes as indicated in Lemma \ref{lem.P2} for $m\in(0,m_s)$ will become a decisive fact in this section. Knowing that $K\to\infty$ implies $C_K\to0$, we begin from the analysis of the limit system obtained by just letting $C_K=0$
\begin{equation}\label{PPSystLim}
\left\{\begin{array}{ll}\dot{U}=C_mUV,\\ \dot{V}=-V^2-C_sV+1-U^{(p-m)/(1-m)},\end{array}\right.
\end{equation}
\begin{lemma}\label{lem.cycles}
The dynamical system \eqref{PPSystLim} does not have limit cycles.
\end{lemma}
\begin{proof}
We use Dulac's Criteria \cite[Theorem 2, Section 3.9]{Pe} taking a generic function $U^{a}$, with $a$ to be determined later, as "integrating factor". We compute the divergence of the vector field obtained by multiplying the vector field of the system \eqref{PPSystLim} by $U^a$ to get
\begin{equation*}
\begin{split}
\frac{\partial}{\partial U}(C_mU^{a+1}V)&+\frac{\partial}{\partial V}\left(-U^aV^2-C_sU^aV+U^a-U^{a+(p-m)/(1-m)}\right)\\
&=(a+1)C_mU^aV-2U^aV-C_sU^a=-C_sU^a<0,
\end{split}
\end{equation*}
by choosing $a$ such that $a=(2-C_m)/C_m$ and taking into account that $C_s>0$ for $m\in(0,m_s)$. Since the divergence has always the same sign, Dulac's Criteria concludes the proof.
\end{proof}
We easily notice that the critical points $P_0$, $P_1$ and $P_2$ and the critical points at infinity remain the same and have similar local analysis in the system \eqref{PPSystLim} as the analysis we did in Section \ref{sec.local}. In particular, we can consider $l_1(\infty)$ to be the unique orbit entering the saddle point $P_1$ in the limit system \eqref{PPSystLim}. We infer from Lemma \ref{lem.cycles}, the local analysis of the points and the Poincar\'e-Bendixon's theory \cite[Section 3.7]{Pe} that the orbit $l_1(\infty)$ must come from a critical point among $Q_2$ or $P_0$. We next show that the latter is impossible.
\begin{lemma}\label{lem.nogood.limit}
There cannot be a trajectory connecting $P_0$ and $P_1$ in the system \eqref{PPSystLim}.
\end{lemma}
\begin{proof}
Assume for contradiction that there exists such a connection, which means that $l_0(\infty)=l_1(\infty)$, where $l_0(\infty)$ denotes the orbit going out of $P_0$ in the system \eqref{PPSystLim}. Let $U_0(\infty)$ be the $U$-coordinate of the first point at which $l_0(\infty)$ crosses the $U$ axis. Since it is obvious that $U_0(\infty)<\infty$ (as the orbit $l_0(\infty)$ goes to $P_1$), we deduce by monotonicity and continuity with respect to the parameter $C_K$ in the system \eqref{PPSystRen} that for $K$ very large $1<U_1(K)<U_0(K)<\infty$. Here it is essential that \emph{for $K$ very large, the point $P_2$ is stable}. It thus follows, for such $K$ sufficiently large, that the orbit $l_1(K)$ must go out of a limit cycle which lies in the region limited by the orbit $l_0(K)$ and the $V$ axis, since it cannot go out of $P_2$ and the orbit $l_0(K)$ becomes a barrier for $l_1(K)$ that cannot be crossed. We next prove that this scenario is impossible using again Dulac's Criteria, with exactly the same multiplying function $U^a$, $a=(2-C_m)/C_m$ as in the proof of Lemma \ref{lem.cycles}. In this case, the divergence of the vector field obtained from the one of the system \eqref{PPSystRen} multiplied by $U^a$ is obtained from the previous one by adding the influence of the term with $C_K$, namely
\begin{equation}\label{interm11}
-C_sU^a+C_kU^{a+1}=C_KU^a\left(U-\frac{C_s}{C_K}\right),
\end{equation}
which for $K>0$ sufficiently large is negative in the whole strip $\{0<U<2U_0(\infty)\}$ (recalling that $C_s>0$ for any $m\in(0,m_s)$). This implies that there are no limit cycles included in the strip $\{0<U<2U_0(\infty)\}$ and it is obvious that a bigger limit cycle must cross $l_0(K)$, which is a contradiction. We conclude that the connection $P_0$-$P_1$ is impossible in the limit system \eqref{PPSystLim}.
\end{proof}
We thus conclude as an outcome of Lemma \ref{lem.cycles}, Lemma \ref{lem.nogood.limit} and the local analysis near $P_2$ done in Lemma \ref{lem.P2} that the orbit $l_1(\infty)$ in the limit system \eqref{PPSystLim} goes out of $Q_2$, while the orbit $l_0(\infty)$ must stay inside the region limited by the $V$ axis and the orbit $l_1(\infty)$ and thus enter the (stable point) $P_2$. We are now in a position to prove that the same holds true for the system \eqref{PPSystRen} with $K$ very large.
\begin{lemma}\label{lem.large}
There exists $K_0>0$ sufficiently large such that for any $K\in(K_0,\infty)$, the orbit $l_0(K)$ enters the critical point $P_2$ and $U_0(K)<U_1(K)$ for $K\in(K_0,\infty)$.
\end{lemma}
\begin{proof}
From the previous discussion and the continuity with respect to the parameter $C_K$ in the system \eqref{PPSystRen} near $C_K=0$, we get that for $K$ sufficiently large the orbit $l_1(K)$ also goes out of $Q_2$, since $Q_2$ is an unstable node. The orbit $l_0(K)$ cannot cross the orbit $l_1(K)$ and thus must remain forever in the region limited by the $V$ axis and the orbit $l_1(K)$, which immediately implies $U_0(K)<U_1(K)$. Moreover, by the same consideration of non-existence of limit cycles in big strips obtained in the proof of Lemma \ref{lem.nogood.limit}, the fact that for $K$ large $P_2$ is a stable node or focus and the Poincar\'e-Bendixon's theory we further deduce that $l_0(K)$ enters $P_2$, as stated.
\end{proof}
At this point, before ending the proof of Theorem \ref{th.1}, let us plot in Figure \ref{fig3} the extremal configurations of the phase plane associated to the system \eqref{PPSystRen}, that is, first for $K>0$ sufficiently small and then for $K\in(K_0,\infty)$ as proved in Lemma \ref{lem.large}.

\begin{figure}[ht!]
  \begin{center}
  \subfigure[$K>0$ small]{\includegraphics[width=7.5cm,height=6cm]{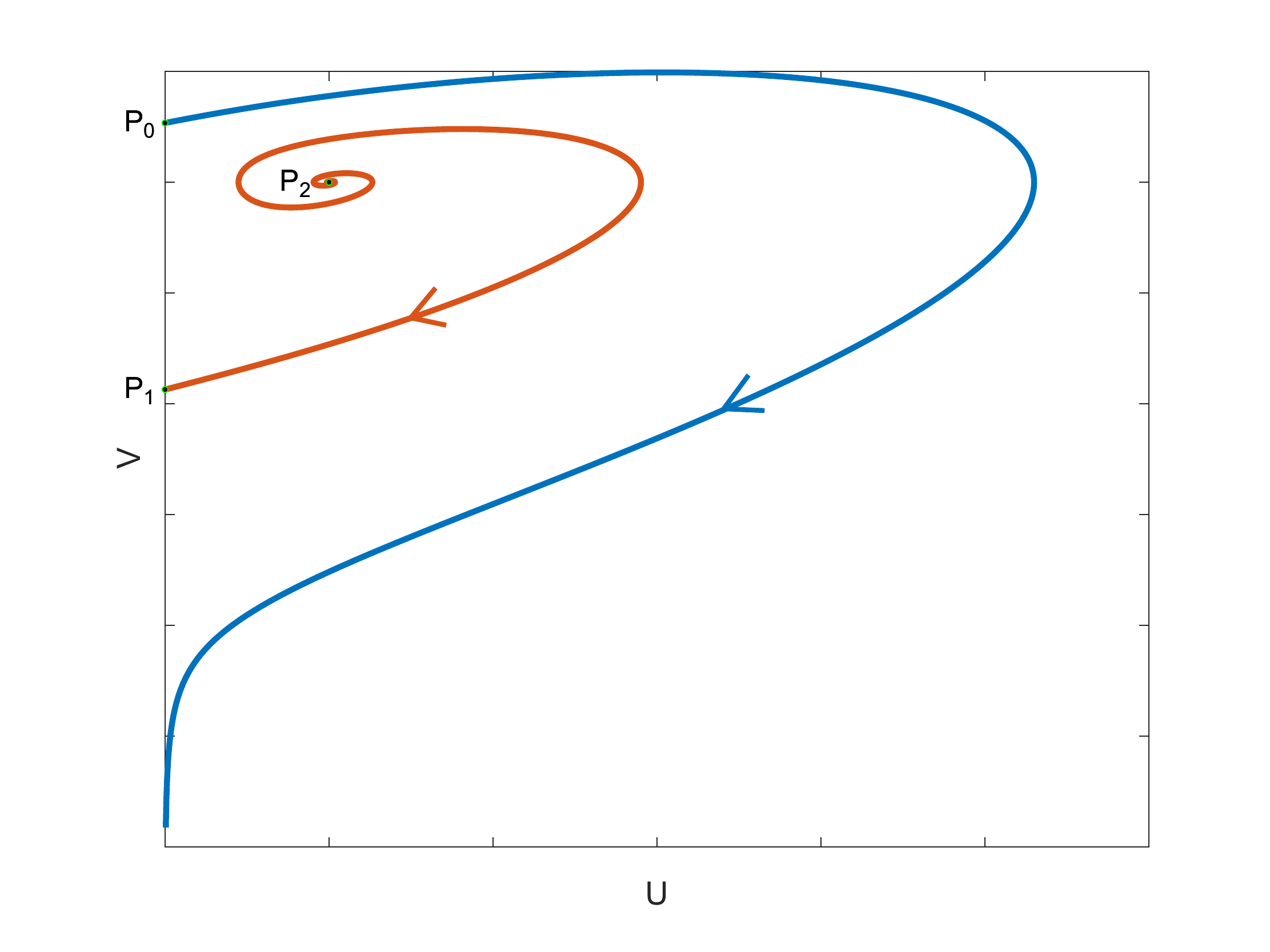}}
  \subfigure[$K>0$ large]{\includegraphics[width=7.5cm,height=6cm]{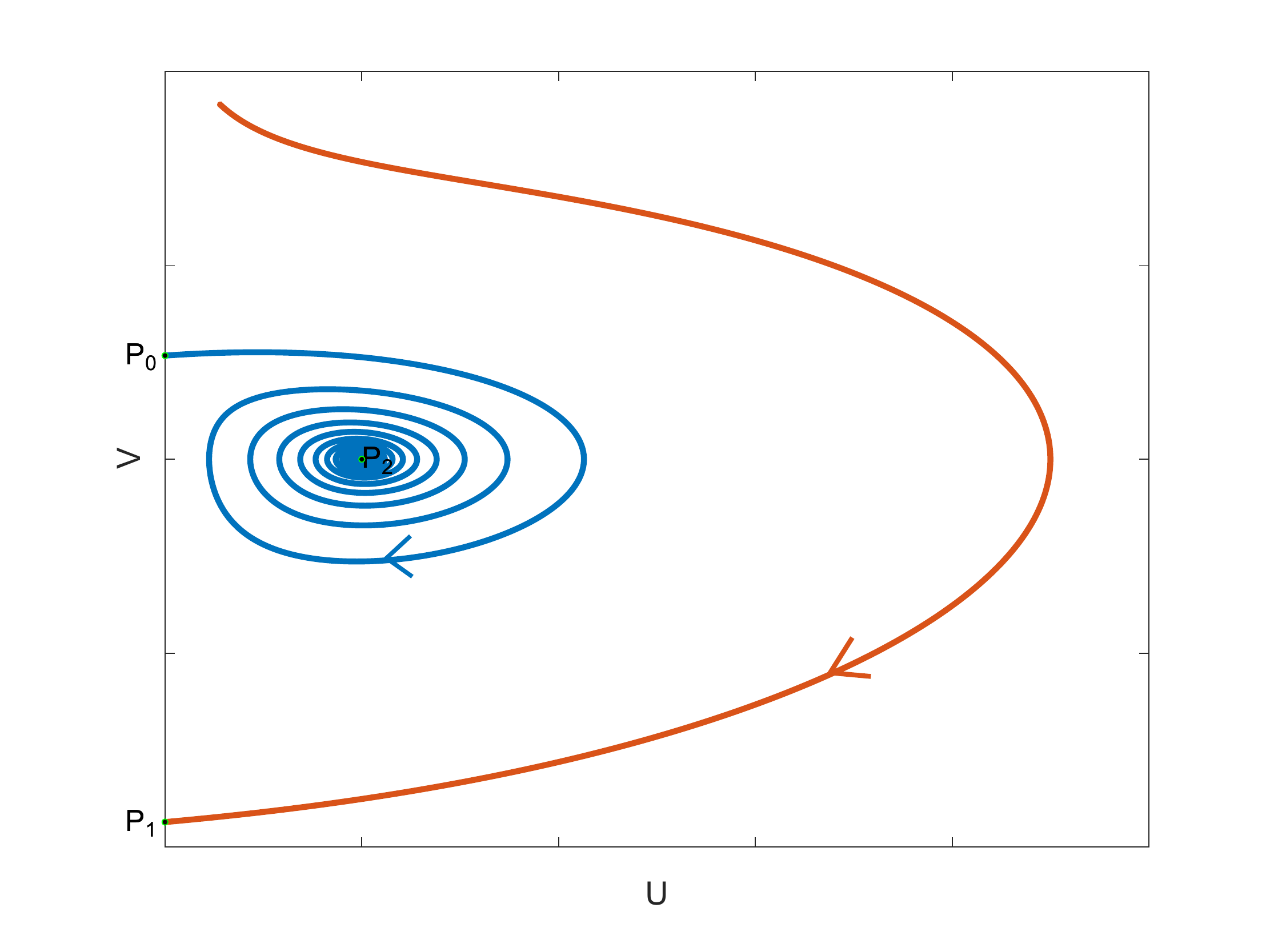}}
  \end{center}
  \caption{Trajectories in the phase space for $K>0$ small and $K>0$ large. Numerical experiment for $N=3$, $m=0.1$, $p=2$ and $K=0.5$, respectively $K=10$}\label{fig3}
\end{figure}

We are now in a position to complete the proof of Theorem \ref{th.1} for any $m\in(0,m_s)$ and any $p>1$.
\begin{proof}[Proof of Theorem \ref{th.1} for $0<m<m_s$]
With the previous notation, introduce the following function
$$
g(K):=U_0(K)-U_1(K),
$$
with the convention that we allow $g(K)=\infty$ if $U_0(K)=\infty$. On the one hand, we have just proved in Lemma \ref{lem.large} that $g(K)<0$ for any $K\in(K_0,\infty)$. On the other hand, Lemma \ref{lem.small} gives that $U_1(K)=1$ for any $K\in(0,K_1)$ and thus, by considerations of flow, $U_0(K)>1$ for $K\in(0,K_1)$ and $g(K)>0$ in this interval. It is obvious that $g$ is a continuous function at least while $g(K)<\infty$, by a standard continuity argument with respect to the parameter, and Lemmas \ref{lem.monot} and \ref{lem.strict} give that $g$ is a strictly decreasing function once $g(K)<\infty$. It thus follows by Bolzano's theorem that there exists a unique value of $K>0$ for which $g(K)=0$, meaning that $U_0(K)=U_1(K)$ and, since this is fulfilled at some point $U_0(K)>1$, thus not a critical point, we infer that the two orbits should coincide, realizing a unique connection between $P_0$ and $P_1$, as claimed.
\end{proof}
We show in Figure \ref{fig4} the outcome of numerical experiments confirming the reversed monotonicity of the values $U_0(K)$ and $U_1(K)$ proved in Lemmas \ref{lem.monot} and \ref{lem.strict} for different shooting parameters $K$ and the formation of the critical orbit connecting $P_0$ and $P_1$.

\begin{figure}[ht!]
  \begin{center}
  \includegraphics[width=11cm,height=7.5cm]{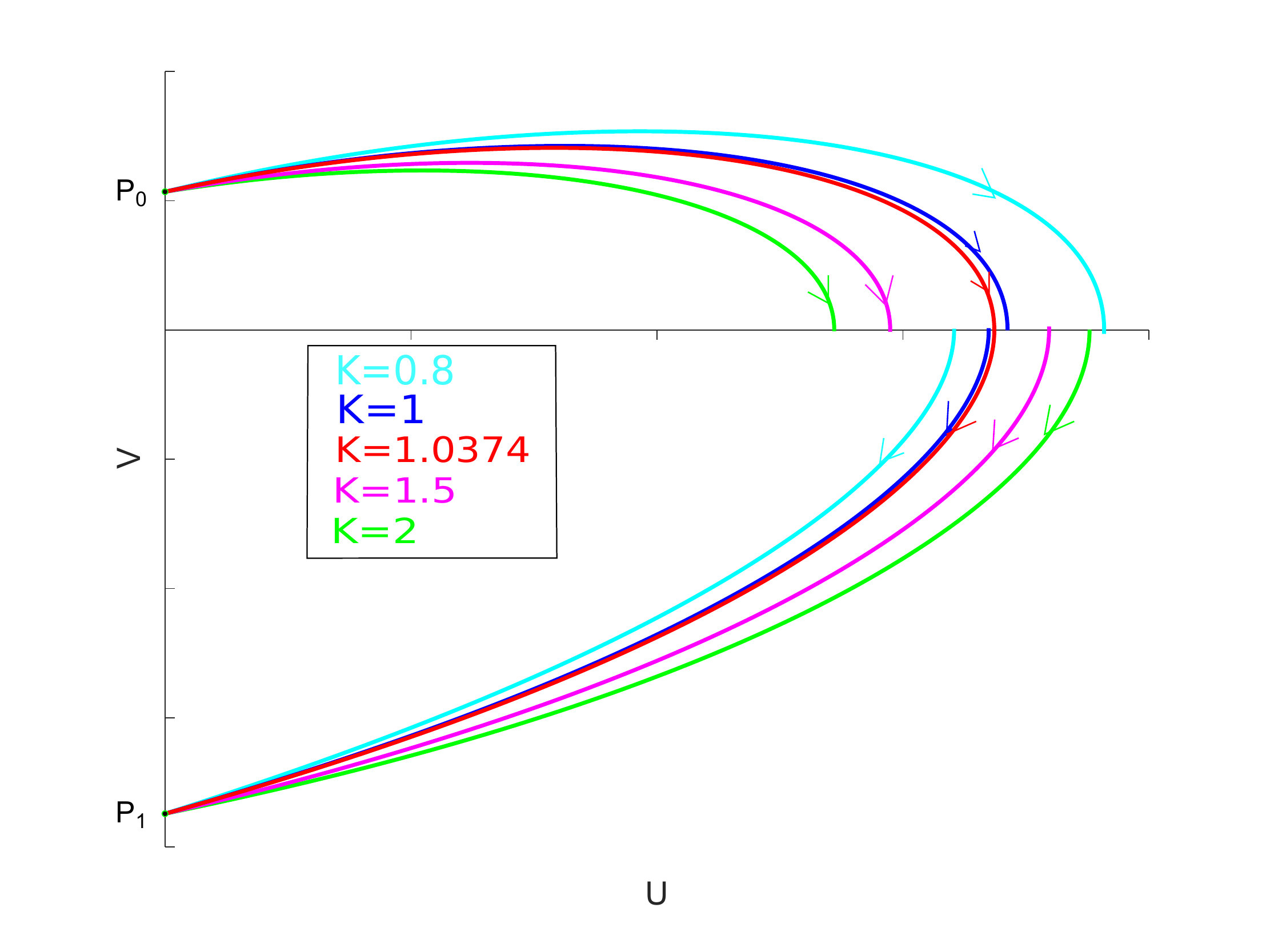}
  \end{center}
  \caption{Several orbits from $P_0$ and entering $P_1$ for different values of $K$. Experiment for $N=3$, $m=0.1$ and $p=2$}\label{fig4}
\end{figure}

\noindent \textbf{Remark.} Apart from the anomalous eternal solution obtained as the unique connection $P_0$-$P_1$, we infer from Lemma \ref{lem.large} that there exist infinitely many orbits connecting $P_0$ to the critical point $P_2$ (for $K\in(K_0,\infty)$ at least). These orbits contain profiles such that $f(0)=A>0$, $f'(0)=0$ and
$$
f(\xi)\sim C\xi^{-2/(1-m)}, \qquad {\rm as} \ \xi\to\infty.
$$
These profiles give rise to solutions that are not integrable as $|x|\to\infty$, but they are "eternal" analogous for Eq. \eqref{eq1} to the \emph{pseudo-Barenblatt profiles} for the subcritical fast diffusion \eqref{FDE} whose relevance for the large time behavior of the subcritical fast diffusion equation has been emphasized in the well-known paper \cite{BBDGV}.

\subsection{Proof of Theorem \ref{th.1} for $m\in(m_s,m_c)$}\label{subsec.negative}

The plan of this section is to prove that, on the one hand, there are no connections $P_0$-$P_1$ and corresponding profiles with $\alpha>0$ and, on the other hand, by changing the signs of the self-similar exponents we obtain a unique good anomalous solution with $\alpha<0$, as stated in Theorem \ref{th.1}. We begin with the first goal.
\begin{lemma}\label{lem.nonposit}
Fix $N\geq3$, $m\in(m_s,m_c)$ and $p>1$. Then there is no connection between the saddle points $P_0$ and $P_1$ in the system \eqref{PPSystRen}.
\end{lemma}
\begin{proof}
The fundamental difference with respect to the case $m\in(0,m_s)$ is that, for $m\in(m_s,m_c)$, the point $P_2$ is always an unstable node or focus, for any $K>0$. Assume for contradiction that there exists an orbit connecting $P_0$ to $P_1$ (that is, $l_0(K)=l_1(K)$) for some $K>0$. Take then $K'>K$ and closer to $K$. By monotonicity, the orbits $l_0(K')$ and $l_1(K')$ cross the $U$ axis in the order $U_0(K')<U_0(K)=U_1(K)<U_1(K')$. Moreover, when $m\in(m_s,m_c)$ we infer from \eqref{coefficients} that $C_s<0$, hence the same proof of non-existence of limit cycles in large strips done with the aid of Dulac's Criteria in the proof of Lemma \ref{lem.nogood.limit} gives that there are no limit cycles at all, in the whole plane. Indeed, the outcome of \eqref{interm11} is now
$$
C_KU^a\left(U-\frac{C_s}{C_K}\right)>0, \qquad {\rm for} \ {\rm any} \ U>0.
$$
It thus follows that the orbit $l_1(K')$ comes from the unstable node $Q_2$ at infinity and the orbit $l_0(K')$ must remain forever in the region limited by the $V$ axis and by the orbit $l_1(K')$. But this is a contradiction to the Poincar\'e-Bendixon theory \cite[Section 3.7]{Pe}, since $l_1(K')$ cannot either end in a limit cycle (there is none) or at $P_2$, which is now an unstable point.
\end{proof}
Let us now prove the existence of such a connection but with exponents $\alpha<0$ and $\beta>0$ as claimed in Theorem \ref{th.1}. To this end, we have to adapt our systems \eqref{PPSyst} and then \eqref{PPSystRen} to the case $\alpha<0$. Let us thus start from changing the ansatz of the form of the solutions by letting now
$$
u(x,t)=e^{-\alpha t}f(|x|e^{\beta t}), \qquad \alpha>0,
$$
obtaining thus after straightforward calculations the same equation as \eqref{ODE} but with the signs in front of the terms involving $\alpha$ and $\beta$ changed. By doing exactly the same change of variable \eqref{change}, this translates into a change of sign in two terms from the equation for $\dot{Y}$, more precisely
\begin{equation}\label{PPSyst.bis}
\left\{\begin{array}{ll}\dot{X}=X(2-(m-1)Y),\\ \dot{Y}=-mY^2-(N-2)Y-2X-(1-m)XY-KX^{(m-p)/(m-1)},\end{array}\right.
\end{equation}
and furthermore, with the same change of variable \eqref{change2} we get the system
\begin{equation}\label{PPSystRen.bis}
\left\{\begin{array}{ll}\dot{U}=C_mUV,\\ \dot{V}=-V^2-C_sV+1-C_{K}UV-U^{(p-m)/(1-m)},\end{array}\right.
\end{equation}
with the same values for its coefficients as in \eqref{coefficients}. Notice that the only difference of \eqref{PPSystRen.bis} with respect to \eqref{PPSystRen} is a change of the sign of the term involving $C_K$. Let us set now $\overline{V}=-V$ and a new independent variable $\eta^*=-\overline{\eta}$ in the system \eqref{PPSystRen.bis}. It is immediate to see that, in variables $(U,\overline{V})$ and taking derivatives with respect to $\eta^*$, we obtain exactly the same system \eqref{PPSystRen} with the only change that $\overline{C_s}=-C_s$. A careful inspection of the proofs in Subsections \ref{subsec.monot} and \ref{subsec.large} shows that in fact the precise values of the coefficients $C_m$, $C_s$, $C_K$ and the power $(p-m)/(1-m)$ of $U$ are completely irrelevant, the analysis of the dynamical system is completely independent provided that the coefficients satisfy the conditions $C_m>0$, $C_s>0$, $C_K=AK^{-(p-m)/(1-m)}$ with $A>0$, and the power of $U$ any number larger than one. Thus, since the analysis in Subsection \ref{subsec.small} is anyway valid for every $m\in(0,m_c)$, we can repeat step by step the analysis with our new set of coefficients $C_m$, $\overline{C_s}$, $C_K$ and the new power $(p-m)/(1-m)$ with $m\in(m_s,m_c)$ exactly along the same lines as the analysis done before, the only change being that
$$
\overline{C_s}=-C_s=\frac{(N+2)(m-m_s)}{\sqrt{2mN(m_c-m)}}>0, \qquad {\rm since} \ m\in(m_s,m_c).
$$
This concludes the existence of a good anomalous solution with $\alpha<0$ when $m\in(m_s,m_c)$.

\noindent \textbf{Remark.} The importance of the fact that $C_s>0$ is illustrated also by the following calculation. The eigenvalues (either real or complex) of the linearization near the point $P_2=(1,0)$ in the system \eqref{PPSystRen} (and similarly \eqref{PPSystRen.bis} with $\overline{C_s}=-C_s$) are
$$
\lambda_{1,2}=\frac{m(C_K-C_s)\pm\sqrt{m^2(C_K-C_s)^2-4m(p-m)}}{2m},
$$
thus it is essential to have $C_s>0$ in order for the term $C_K-C_s$ to change sign with $K\in(0,\infty)$, allowing thus the change of the critical point $P_2$ from unstable to stable, which was fundamental in the proofs.

\subsection{The explicit case $m=m_s$. Stationary solutions}\label{subsec.stat}

We are left with the case $m=m_s$, for which the analysis performed in Subsections \ref{subsec.large} and \ref{subsec.negative} together with the expected continuity of the exponents $\alpha$ with respect to $m$ give us the idea that we have to look for stationary solutions, that is, with $\alpha=\beta=0$. With this ansatz and recalling that $m=m_s$, we readily get that $C_s=C_K=0$ in the system \eqref{PPSystRen}. We then find that this system becomes integrable by letting
$$
\frac{dV}{dU}=\frac{1-V^2-U^{(p-m)/(1-m)}}{C_mUV},
$$
and getting by direct integration (and putting the integration constant to zero) the explicit curve
\begin{equation}\label{curve.ms}
V^2=1-\frac{2m}{m+p}U^{(p-m)/(1-m)},
\end{equation}
which is indeed an orbit connecting the critical points $P_0=(0,1)$ and $P_1=(0,-1)$ (since $C_s=0$). Let us further notice that the monotonicity arguments in Subsection \ref{subsec.monot} remain valid for $m=m_s$, leading to the uniqueness of the orbit given in \eqref{curve.ms}. With this, Theorem \ref{th.1} is fully proved in this case.

However, despite the fact that the curve \eqref{curve.ms} cannot be easily integrated in terms of profiles, we can still obtain explicit formulas for the stationary solutions contained in the orbit \eqref{curve.ms}, which are exactly equal to their profile since $\alpha=\beta=0$.
\begin{proposition}\label{prop.expl}
The stationary solutions contained in the orbit \eqref{curve.ms} have the explicit form
\begin{equation}\label{expl.stationary}
u(x)=\left[\frac{(N^2-4)(p+m_s)D}{(1+D|x|^{L})^2}\right]^{1/(p-m_s)}, \ L=\frac{(N+2)p-(N-2)}{2}, \ D>0 \ {\rm free \ constant}.
\end{equation}
\end{proposition}
\begin{proof}
It is easy to check directly from \eqref{expl.stationary} that, since $L>1$, $u(0)>0$, $u'(0)=0$ and $u(x)$ has the expected decay
$$
u(x)\sim C|x|^{-(N+2)}=C|x|^{-(N-2)/m_s}, \qquad {\rm as} \ |x|\to\infty,
$$
hence the functions \eqref{expl.stationary} belong to the orbit \eqref{curve.ms} for any $D>0$. However, it is rather instructive to be fair with the reader and explain in the next lines how we actually got to the expression in \eqref{expl.stationary}, since it cannot be done directly from Eq. \eqref{ODE} in an obvious way. Thus, we use the following transformation
\begin{equation}\label{interm20}
w(r)=r^{2/(1-m)}u(r), \qquad y=\ln\,r, \qquad r=|x|,
\end{equation}
which is a particular case of the more general change of variable introduced in \cite[Section 6]{IS21b} to obtain the following equation
\begin{equation}\label{interm21}
0=(w^m)_{yy}-\frac{2mN(m_c-m)}{(m-1)^2}w^m+w^p,
\end{equation}
which is the stationary counterpart of a general Fisher-type equation studied in \cite{SH05, HS14}. Here and in the next lines, the subscripts indicate derivatives with respect to the variable $y$. We can then multiply by $(w^m)_y$ in \eqref{interm21} and integrate to obtain
\begin{equation}\label{interm22}
\frac{1}{2}\left[(w^m)_y\right]^2-\frac{mN(m_c-m)}{(m-1)^2}(w^m)^2+\frac{m}{p+m}w^{p+m}=C.
\end{equation}
Since
$$
w(r)\sim r^{2/(1-m)+(2-N)/m}=r^{N(m-m_c)/m(1-m)}\to0, \qquad {\rm as} \ r\to\infty
$$
and
$$
(w^m)_y\sim r^{N(m-m_c)/(1-m)}\to0, \qquad {\rm as} \ r\to\infty,
$$
we infer that $C=0$ in \eqref{interm22}. We further introduce a new function by setting
$$
g(y)=\frac{(m-1)^2}{2N(p+m)(m_c-m)}w(y)^{p-m},
$$
and \eqref{interm22} writes in term of $g$ as the following easy to integrate differential equation
$$
g_y=\pm Lg\sqrt{1-g}, \qquad L=\frac{p-m_s}{1-m_s}\sqrt{\frac{2N(m_c-m_s)}{m_s}}=\frac{2(p-m_s)}{1-m_s}.
$$
We obtain by integration that
\begin{equation}\label{interm23}
g(y)=\frac{1}{\left[\cosh\left(-\frac{L}{2}(C+y)\right)\right]^2}, \qquad C\in\real \ {\rm free \ constant}.
\end{equation}
Starting from \eqref{interm23} and undoing the transformation in \eqref{interm20} we reach after some straightforward calculations the expression \eqref{expl.stationary}.
\end{proof}

\noindent \textbf{Remark.} This stationary behavior for $m=m_s$ expresses once more the perfect balance between the fast diffusion and the weighted reaction in Eq. \eqref{eq1}. We recall here that the fast diffusion equation Eq. \eqref{FDE} also has explicit solutions for $m=m_s$, related to the geometrical Yamabe problem, but these solutions present finite time extinction \cite[Section 7.2]{VazSmooth}. We raise an \textbf{open problem} connected to these solutions at the end of the present paper.

\subsection{Non-existence for $m\in[m_c,1)$}\label{subsec.nonex}

Let us consider now $m\geq m_c$. This part is now easy by the well-established theory. Indeed, if $m>m_c$ we can reformulate the problem by letting $\sigma>0$ free and expressing $p$ in terms of $\sigma$ from \eqref{sig} as a "critical value" to get
$$
p(\sigma)=1+\frac{\sigma(1-m)}{2}.
$$
Thus, recalling the value of the Fujita-type exponent $p_{F,\sigma}$ in \eqref{Fujita_sigma} we get
$$
p_{F,\sigma}-p(\sigma)=\frac{\sigma+2}{2}\left(\frac{2}{N}+m-1\right)=\frac{\sigma+2}{2}(m-m_c)>0,
$$
whence $1<p(\sigma)<p_{F,\sigma}$ for any $\sigma>0$ and $m\in(m_c,1)$. Thus there cannot exist any "eternal" solution to \eqref{eq1}, as all the solutions to Eq. \eqref{eq1} blow up in finite time according to \cite{Qi98}. For $m=m_c$, we do the following transformation in radial variables (with $r=|x|$)
$$
y=\ln\,r, \qquad w(y,t)=r^{2/(1-m)}u(r,t),
$$
which is a particular case of the general transformation introduced in \cite[Section 6]{IS21b}, leading in our case to the following equation
\begin{equation}\label{eq2}
w_t=(w^m)_{yy}+\frac{2m}{m-1}(w^m)_y+w^p,
\end{equation}
and the eternal self-similar solutions to Eq. \eqref{eq1} are mapped into traveling wave solutions to Eq. \eqref{eq2}, as shown in \cite[Section 6]{IS21b}. But Eq. \eqref{eq2} is a particular case of the more general equation
\begin{equation}\label{eq2.gen}
w_t=a(w^{m})_{yy}+b(w^{m})_y+kw^p,
\end{equation}
which is analyzed in \cite{dPS00}. In particular, it is shown there that Eq. \eqref{eq2.gen} does not admit any traveling waves if $a>0$, $p>1$ and $k>0$, which is exactly our case (with $k=1$, $a=1$ and $p>1$). Thus there are no eternal self-similar solutions to Eq. \eqref{eq1} for $m=m_c$, completing the analysis. The non-existence for the critical case $m=m_c$ can be also seen from the phase plane: indeed, since the critical point $P_2$ disappears for $m=m_c$, an orbit connecting $P_0$ to $P_1$ would lead to a contradiction with the Poincar\'e-Bendixon Theorem.

\section{Some explicit connections in the phase plane and self-maps}\label{sec.appendix}

In this final section we gather several facts that complete the study of Eq. \eqref{eq1}, such as explicit or semi-explicit solutions (identified as explicit orbits in the phase plane), and a self-map of the equation. Most of these explicit solutions or trajectories of the phase plane are obtained when $m+p=2$.

\medskip

\noindent \textbf{Explicit good orbits connecting $P_0$ to $P_1$}. Let us consider $m$, $p$ such that $m+p=2$. We construct below some explicit saddle-saddle connections in the phase plane associated to the system \eqref{PPSyst}. Let us start with a rotation such that the eigenvector tangent to the orbit $l_0(K)$ going out of $P_0$ is mapped on the $Y$-axis. That is done by introducing
$$
W:=-Y+\frac{2}{N}X
$$
and obtain a new, equivalent system
\begin{equation}\label{PPSyst.exp}
\left\{\begin{array}{ll}\dot{X}=(m-1)XW-\frac{2(m-1)}{N}X^2+2X,\\ \dot{W}=mW^2-(N-2)W-\frac{(N+2)(m-m_s)}{N}XW+\frac{KN^2+2N(m-m_c)}{N^2}X^2.\end{array}\right.
\end{equation}
The idea is to look for explicit solutions of the system \eqref{PPSyst.exp} in the particular form
\begin{equation}\label{good.orbits}
X=aW^{1/2}+bW, \qquad a, b \ {\rm to \ be \ determined}
\end{equation}
and show that for suitable choices of $a$ and $b$, the orbits in \eqref{good.orbits} describe saddle-saddle connections between $P_0$ and $P_1$, thus containing good profiles. To this end, the main idea is to calculate the direction of the flow of the system \eqref{PPSyst.exp} on the curves of the form \eqref{good.orbits} and ask it to be identically zero. These calculations are rather tedious and have been done with the aid of a symbolic calculation program. We get that the flow is given by the following expression (depending on $W$)
\begin{equation}\label{flow}
F(W)=-\frac{1}{2N^2}\left(A_1W^2+A_2W^{3/2}+A_3W+A_4W^{1/2}\right),
\end{equation}
where $A_i$, $i=1,4$, are explicit expressions depending on $m$, $N$, $K$, $a$ and $b$ (recall that we have $m+p=2$) whose expressions will be introduced one by one below. We require all these four coefficients to be zero and obtain some values for $a$ and $b$. We start with $A_4$:
$$
A_4=(-KN^2-2mN+2N-4)a^3+N^2(N+2)a=0,
$$
from where we deduce the value of $a$ by letting
\begin{equation}\label{interm12}
a=\sqrt{\frac{N^2(N+2)}{KN^2+2N(m-m_c)}}, \qquad {\rm provided} \ K>\frac{2(m_c-m)}{N}.
\end{equation}
We go now to the coefficient $A_3$, which writes
$$
A_3=a^2[-4b(KN^2+2N(m-m_c))+N(mN-N-2m+6)]+2N^3b
$$
and equate $A_3=0$ in terms of $b$, after substituting $a^2$ by its expression in \eqref{interm12}, to get
\begin{equation}\label{interm13}
b=\frac{N(mN^2-N^2+4N-4m+12)}{2(KN^3+4KN^2+2N^2m-2N^2+8mN-4N+16)}.
\end{equation}
We further go to the expression of $A_1$ to find out the precise value of $K$. We have
$$
A_1=-2b^3(KN^2+2N(m-m_c))+2b^2N(mN-N+4)-2N^2b=0,
$$
from which, after substituting $b$ from \eqref{interm13} we obtain the precise value of the parameter $K$ for which the orbits exist
\begin{equation}\label{interm14}
K=\frac{[N(N+8)(m-1)+4(m+1)][N^2(m-1)-4(m+1)]}{4N^2(N+4)^2}
\end{equation}
and then the value of $b$ after replacing this value of $K$ in \eqref{interm13}
\begin{equation}\label{interm15}
b=\frac{2N(N+4)}{mN^2-N^2+8mN-4N+4m+20}.
\end{equation}
It is easy to check that $K>0$ in \eqref{interm14} for any $m\in(0,m_c)$, as both factors in the numerator of its formula are negative in this range. Moreover, the compatibility condition given in \eqref{interm12} to insure the existence of $a$ becomes
$$
\frac{2(m_c-m)}{N}-K=-\frac{(N+2)(mN-N-2m+6)[m(N^2+8N+4)-N^2-4N+20]}{4N^2(N+4)^2}<0,
$$
which is fulfilled if either $m<m_1$ or $m>m_2$, where
\begin{equation}\label{interm16}
m_1=\frac{N-6}{N-2}, \qquad m_2=\frac{N^2+4N-20}{N^2+8N+4}.
\end{equation}
Let us notice that $m_2<m_s$ and that $m_1>0$ if and only if $N\geq7$. We are now left with the second coefficient $A_2$ in \eqref{flow}, which after replacing $K$, $b$ and $a$ with their expressions in \eqref{interm14}, \eqref{interm15} and \eqref{interm12} respectively, gives
$$
A_2=\frac{N^2[(N^2+8N+4)m^2-(2N^2-16)m+(N-2)(N-6)]}{(N^2+8N+4)m-N^2-4N+20}=0.
$$
Defining
\begin{equation}\label{interm17}
f(m)=(N^2+8N+4)m^2-(2N^2-16)m+(N-2)(N-6),
\end{equation}
we readily find that $f(m_1)<0$, $f(m_2)<0$ and $f(m_s)<0$, thus we infer that $m_1$, $m_2$ and $m_s$ belong to the interval $(m_3,m_4)$ of its roots, which are given by
\begin{equation}\label{interm18}
m_3=\frac{N^2-8-4\sqrt{2N^2-4N+1}}{N^2+8N+4}, \qquad m_4=\frac{N^2-8+4\sqrt{2N^2-4N+1}}{N^2+8N+4}.
\end{equation}
Let us further notice (by easy calculations that we omit) that $m_3>0$ if and only if $N>6$ and that $m_2<m_c$ for every $N$. Moreover, we remark that
$$
\frac{N-2}{m}-\frac{a^2}{b^2}=-\frac{f(m)}{m(mN-N-2m+6)}=0,
$$
provided $m=m_3$ or $m=m_4$, where $f(m)$ is defined in \eqref{interm17}. The latter shows that the orbit we constructed in the phase plane enters the critical point $P_1$.

Putting everything together, the construction is done through the following process: pick any dimension $N\geq7$ and then let $m=m_3\in(0,m_s)$ given by \eqref{interm18}, $K>0$ given by \eqref{interm14}, $p=2-m_3$, $b$ given by \eqref{interm15} and $a$ given by \eqref{interm12}. With these choices, we get an explicit good connection $P_0$-$P_1$ for any such dimension $N\geq7$ given by \eqref{good.orbits}. We can also notice along the same lines that if we choose $m=m_4\in(m_s,m_c)$ we get a good connection in the phase plane obtained for the case $\alpha<0$ and described in Subsection \ref{subsec.negative}.

\medskip

\noindent \textbf{Other explicit profiles that are not contained in a connection $P_0$-$P_1$}. Apart from the saddle-saddle connections $P_0$-$P_1$ constructed above, we can give some more examples of orbits and profiles connecting the critical points in the phase plane associated to the system \eqref{PPSyst}.

$\bullet$ There exists an explicit solution to \eqref{FDE}
\begin{equation}\label{sol.P2}
f(\xi)=C\xi^{-2/(1-m)}, \qquad C=\left[\frac{2mN(m_c-m)}{(m-1)^2}\right]^{1/(p-m)},
\end{equation}
which is represented in the phase plane by the critical point $P_2$ itself. Independent of $\alpha$ and $\beta$, this gives rise to a stationary solution $u(x,t)=C|x|^{-2/(1-m)}$, which presents a vertical asymptote at the origin. Such a solution is an analogous for Eq. \eqref{eq1} to the separate variable solution $U(x,t;T)$ to the standard fast diffusion equation Eq. \eqref{FDE} in \cite[Section 5.2.1, p.80]{VazSmooth}, with the noticeable difference that our solution is stationary and the solution to Eq. \eqref{FDE} extinguishes in finite time. This is another illustration of the perfect balance between diffusion and reaction in our equation. A rather similar stationary solution exists for the homogeneous case $\sigma=0$ with $p=1$ and $m<1$ as a limit case of the more general stationary solutions for $p>m$
$$
u(x,t)=C|x|^{2/(m-p)}, \qquad p>\frac{m(N-2)}{N},
$$
given in \cite[Section V.2.2, p.212]{S4}.

$\bullet$ Letting again $m+p=2$, one can look for orbits that are straight lines of the form $Y=aX+b$ in the phase plane. The direction of the flow of the system \eqref{PPSyst} on such a line is given by
\begin{equation}\label{interm19}
G(X)=(-a^2-am+a-K)X^2-[ab(m+1)+aN+b(m-1)-2]X-b(bm+N-2)
\end{equation}
and we wish to have $G(X)\equiv0$. From the last term we deduce that either $b=0$ or $b=-(N-2)/m$. On the one hand, if $b=0$, we infer from equating to zero the other coefficients in \eqref{interm19} that $a=2/N$ and $K=2(m_c-m)/N>0$, thus we get a line $Y=2X/N$. By replacing $X$, $Y$ from \eqref{change} and an easy integration, we obtain the family of explicit profiles
\begin{equation}\label{expl.P0Q4}
f(\xi)=\left[C-\frac{\alpha(1-m)}{2mN}\xi^2\right]^{-1/(1-m)}, \qquad C>0 \ {\rm free \ constant},
\end{equation}
presenting a vertical asymptote and belonging to a straight line connecting $P_0$ to the critical point at infinity $Q_4$ in the phase plane. On the other hand, if $b=-(N-2)/m$, we infer from equating to zero the other coefficients in \eqref{interm19} that
$$
a=\frac{N(m-m_c)}{2m-N+2}, \qquad K=\frac{2N(m_c-m)m^2}{(2m-N+2)^2}
$$
which are both positive if $m<m_c$ and $N\geq3$. We thus get a line $Y=aX+b$ that starts from $P_1$, passes through $P_2$ and then ends at $Q_4$. We obtain thus by integration the following explicit family of profiles
\begin{equation}\label{expl.P1Q4}
f(\xi)=\xi^{-2/(1-m)}\left[\frac{(1-m)\alpha}{2(N-2-2m)}+D\xi^{N(m_c-m)/m}\right]^{-2/(1-m)}, \qquad D\in\real \ {\rm free \ constant},
\end{equation}
where for $D=0$ we recall the stationary solution given in \eqref{sol.P2}, with $D>0$ we get profiles having a vertical asymptote at $\xi=0$ and good tail behavior as $\xi\to\infty$ (the line connecting $P_2$-$P_1$) and with $D<0$ we get a profile with two vertical asymptotes (the line connecting $P_2$-$Q_4$). Such a family of profiles with two vertical asymptotes has been also obtained for Eq. \eqref{eq1} with $\sigma=0$ and $p>m+2/N$ but presenting finite time blow-up, see for example \cite[Figure 5.1,p. 214]{S4}.

\medskip

\noindent \textbf{A self-map of Eq. \eqref{eq1}}. It is straightforward to check that, fixing $m\in(0,m_c)$ and $p>1$, we can obtain exactly the same phase plane associated to the system \eqref{PPSystRen} by equating $C_s$ and $C_K$ for different values of $N$ and $K$, namely
\begin{equation}\label{self.map1}
\overline{N}=\frac{2(N-2-2m)}{N(m_c-m)}, \qquad \overline{K}=K\left[2mN(m_c-m)\right]^{(2-p-m)/(1-m)}.
\end{equation}
We then infer that the change of self-similar exponents from one solution to the other is given by
$$
\overline{\alpha}=\alpha\left[2mN(m_c-m)\right]^{(m+p-2)/(p-m)}, \qquad \overline{\beta}=\beta\left[2mN(m_c-m)\right]^{(m+p-2)/(p-m)},
$$
while from equating the variables of the phase plane system we readily get the following changes for the dependent and the independent variables:
\begin{equation}\label{self.map2}
\overline{\xi}=\xi^{-N(m_c-m)/2m}, \qquad \overline{f}(\overline{\xi})=\left[\frac{2m}{N(m_c-m)}\right]^{2/(p-m)}\xi^{(N-2)/m}f(\xi).
\end{equation}
It is easy to check that, due to the change of dimension in \eqref{self.map1}, the self-map given in \eqref{self.map2} matches the interval $m\in(0,m_s)$ into $m\in(m_s,m_c)$ (in fact, the value of $m$ is the same, but as the dimension changes, also the critical exponents $m_c$ and $m_s$ change according to \eqref{mc} and \eqref{ms}) and it is an inversion, mapping the anomalous solutions between themselves.

Let us finally notice that the self-map to Eq. \eqref{eq1} constructed in \eqref{self.map1}-\eqref{self.map2} is a generalization of an interesting \emph{self-map for the fast diffusion equation} \eqref{FDE} obtained as a particular case of the more general self-maps constructed in \cite[Section 2.1]{IRS13} (taking $\gamma=\tilde{\gamma}=0$ in the notation of the quoted reference) but that seems to have remained unnoticed: for the fast diffusion equation the change of dimension is exactly the same as in \eqref{self.map1}, while the changes of independent variable and function are perfectly similar to the ones in \eqref{self.map2} except for the constant in front of the right-hand side of the change from $\overline{f}$ to $f$. A similar self-map for the porous medium equation or the fast diffusion equation with $m>m_c$ has been introduced in \cite[Section 3]{ISV08}, the algebraic (symbolic) formulas being essentially the same ones as for the range $m\in(0,m_c)$.

\medskip

\noindent \textbf{An open problem.} Related to the solutions in Proposition \ref{prop.expl} and the idea of using in the proof a transformation to a generalized Fisher-type equation analyzed in the short note \cite{HS14}, it is there shown that, in the "neighbor case" of letting $w^q$ with $q<m$ instead of $w^m$ in the reaction part of \eqref{interm22}, the stationary solutions act as a \emph{separatrix between blow-up and extinction}, in the sense that any solution with suitably regular initial condition $u_0(x)$ lying entirely below the stationary solution vanishes in finite time and any solution with suitably regular initial condition $u_0(x)$ lying entirely above the stationary solution blows up in finite time. This allows us to raise the following question, which we believe that is very interesting: is it true, that the anomalous eternal solutions constructed in Theorem \ref{th.1} for any $m\in(0,m_c)$ (or at least, the stationary ones for the explicit case $m=m_s$) also separate for our equation Eq. \eqref{eq1} between solutions that vanish in finite time and solutions that blow up in finite time?

\bigskip

\noindent \textbf{Acknowledgements} A. S. is partially supported by the Spanish project MTM2017-87596-P.

\bibliographystyle{plain}

\end{document}